\documentclass[fleqn, pdftex]{article} %arXiv

\usepackage{amsmath,amssymb,amsthm}
\usepackage{tikz}

\title{The fixed point and the Craig interpolation properties for sublogics of $\mathbf{IL}$}
\author{Sohei Iwata, Taishi Kurahashi and Yuya Okawa}
\date{}

\theoremstyle{plain}
\newtheorem{thm}{Theorem}[section]
\newtheorem{lem}[thm]{Lemma}
\newtheorem{prop}[thm]{Proposition}
\newtheorem{cor}[thm]{Corollary}
\newtheorem{fact}[thm]{Fact}
\newtheorem{prob}[thm]{Problem}
\newtheorem{cl}{Claim}

\theoremstyle{definition}
\newtheorem{defn}[thm]{Definition}

\newcommand{\GL}{\mathbf{GL}}
\newcommand{\CL}{\mathbf{CL}}
\newcommand{\IL}{\mathbf{IL}}
\newcommand{\rank}{\mathrm{rank}}
\newcommand{\Sub}{\mathrm{Sub}}
\newcommand{\PSub}{\mathrm{PSub}}

\newcommand{\G}[1]{\mathbf{L#1}}
\newcommand{\J}[1]{\mathbf{J#1}}
\newcommand{\R}[1]{\mathbf{R#1}}

\newcommand{\seq}[1]{\langle#1\rangle}

\begin{document}

\maketitle

\begin{abstract}
We study the fixed point property and the Craig interpolation property for sublogics of the interpretability logic $\IL$. 
We provide a complete description of these sublogics concerning the uniqueness of fixed points, the fixed point property and the Craig interpolation property. 
\end{abstract}

\section{Introduction}\label{Sec:Intro}

De Jongh and Sambin's fixed point theorem \cite{Sam76} for the modal propositional logic $\GL$ is one of notable results of modal logical investigation of formalized provability. 
For any modal formula $A$, let $v(A)$ be the set of all propositional variables contained in $A$. 
A logic $L$ is said to have the fixed point property (FPP) if for any modal formula $A(p)$ in which the propositional variable $p$ appears only in the scope of $\Box$, there exists a modal formula $B$ such that $v(B) \subseteq v(A) \setminus \{p\}$ and $L \vdash B \leftrightarrow A(B)$. 
De Jongh and Sambin's theorem states that $\GL$ has FPP, and this is understood as a counterpart of the fixed point theorem in formal arithmetic (see \cite{Boo93}). 
Bernardi \cite{Ber76} also proved the uniqueness of fixed points (UFP) for $\GL$. 

A logic $L$ is said to have the Craig interpolation property (CIP) if for any formulas $A$ and $B$, if $L \vdash A \to B$, then there exists a formula $C$ such that $v(C) \subseteq v(A) \cap v(B)$, $L \vdash A \to C$ and $L \vdash C \to B$. 
Smory\'nski \cite{Smo78} and Boolos \cite{Boo79} independently proved that $\GL$ has CIP. 
Smory\'nski also made an important observation that FPP for $\GL$ follows from CIP and UFP.  

The interpretability logic $\IL$ is an extension of $\GL$ in the language of $\GL$ equipped with the binary modal operator $\rhd$, where the modal formula $A \rhd B$ is read as ``$T + B$ is relatively interpretable in $T + A$''. 
It is natural to ask whether $\IL$ also has the properties that hold for $\GL$.
Indeed, de Jongh and Visser \cite{DeJVis91} proved UFP for $\IL$ and that $\IL$ has FPP. 
Also Areces, Hoogland and de Jongh \cite{AHD01} proved that $\IL$ has CIP. 

Ignatiev \cite{Ign91} introduced the sublogic $\CL$ of $\IL$ as a base logic of the modal logical investigation of the notion of partial conservativity, and proved that $\CL$ is complete with respect to relational semantics (that is, regular Veltman semantics). 
Kurahashi and Okawa \cite{KO20} also introduced several sublogics of $\IL$, and showed the completeness and the incompleteness of these sublogics with respect to relational semantics. 

In this paper, we investigate UFP, FPP and CIP for sublogics of $\IL$ shown in Figure \ref{Fig1}. 

\begin{figure}[h]\label{Fig1}
\centering
\begin{tikzpicture}
\node (IL-) at (0,1.2) {$\IL^-$};
\node (IL5) at (3,0) {$\IL^-(\J{5})$};
\node (IL1) at (3,1.2) {$\IL^-(\J{1})$};
\node (IL4) at (3,2.4){$\IL^-(\J{4}_+)$};
\node (IL15) at (6,0){$\IL^-(\J{1}, \J{5})$};
\node (IL45) at (6,1.2){$\IL^-(\J{4}_+, \J{5})$};
\node (IL14) at (6,2.4){$\IL^-(\J{1}, \J{4}_+)$};
\node (IL2) at (6,3.6){$\IL^-(\J{2}_+)$};
\node (IL145) at (9,1.2){$\IL^-(\J{1}, \J{4}_+, \J{5})$};
\node (IL25) at (9,2.4){$\IL^-(\J{2}_+, \J{5})$};
\node (CL) at (9,3.6){$\CL$};
\node (IL) at (12,2.4){$\IL$};
\draw [-] (IL5)--(IL-);
\draw [-] (IL1)--(IL-);
\draw [-] (IL4)--(IL-);
\draw [-] (IL15)--(IL5);
\draw [-] (IL45)--(IL5);
\draw [-] (IL15)--(IL1);
\draw [-] (IL14)--(IL1);
\draw [-] (IL45)--(IL4);
\draw [-] (IL14)--(IL4);
\draw [-] (IL2)--(IL4);
\draw [-] (IL145)--(IL15);
\draw [-] (IL145)--(IL45);
\draw [-] (IL25)--(IL45);
\draw [-] (IL145)--(IL14);
\draw [-] (CL)--(IL14);
\draw [-] (IL25)--(IL2);
\draw [-] (CL)--(IL2);
\draw [-] (IL)--(IL145);
\draw [-] (IL)--(IL25);
\draw [-] (IL)--(CL);
\end{tikzpicture}
\caption{Sublogics of $\IL$}
\end{figure}
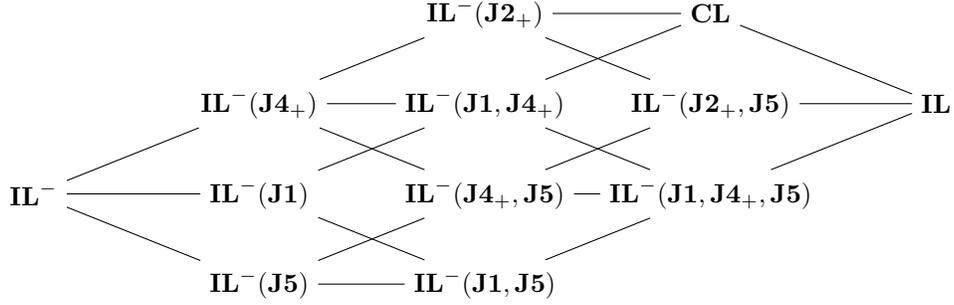

Moreover, for technical reasons, we introduce and investigate the notions of $\ell$UFP and $\ell$FPP that are restricted versions of UFP and FPP with respect to some particular forms of formulas, respectively. 
Table \ref{Tab1} summarizes a complete description of these sublogics concerning $\ell$UFP, UFP, $\ell$FPP, FPP and CIP. 

\begin{table}[h]\label{Tab1}
\centering
\begin{tabular}{|l||c|c|c|c|c|}
\hline
 & $\ell$UFP & UFP & $\ell$FPP & FPP & CIP \\
\hline
\hline
$\IL^-$ & $\checkmark$ & $\times$ & $\times$ & $\times$ & $\times$ \\
\hline
$\IL^-(\J{1})$ & $\checkmark$ & $\times$ & $\times$ & $\times$ & $\times$ \\
\hline
$\IL^-(\J{5})$ & $\checkmark$ & $\times$ & $\times$ & $\times$ & $\times$ \\
\hline
$\IL^-(\J{1}, \J{5})$ & $\checkmark$ & $\times$ & $\times$ & $\times$ & $\times$ \\
\hline
$\IL^-(\J{4}_+)$ & $\checkmark$ & $\checkmark$ & $\times$ & $\times$ & $\times$ \\
\hline
$\IL^-(\J{1}, \J{4}_+)$ & $\checkmark$ & $\checkmark$ & $\times$ & $\times$ & $\times$ \\
\hline
$\IL^-(\J{2}_+)$ & $\checkmark$ & $\checkmark$ & $\times$ & $\times$ & $\times$ \\
\hline
$\CL$ & $\checkmark$ & $\checkmark$ & $\times$ & $\times$ & $\times$ \\
\hline
$\IL^-(\J{4}_+, \J{5})$ & $\checkmark$ & $\checkmark$ & $\checkmark$ & $\times$ & $\times$ \\
\hline
$\IL^-(\J{1}, \J{4}_+, \J{5})$ & $\checkmark$ & $\checkmark$ & $\checkmark$ & $\times$ & $\times$ \\
\hline
$\IL^-(\J{2}_+, \J{5})$ & $\checkmark$ & $\checkmark$ & $\checkmark$ & $\checkmark$ & $\checkmark$ \\
\hline
$\IL$ & $\checkmark$ & $\checkmark$ \cite{DeJVis91} & $\checkmark$ & $\checkmark$ \cite{DeJVis91} & $\checkmark$ \cite{AHD01} \\
\hline
\end{tabular}
\caption{$\ell$UFP, UFP, $\ell$FPP, FPP and CIP for sublogics of $\IL$}
\end{table}

The paper is organized as follows. 
In Section \ref{Sec:UFP}, we show that UFP holds for extensions of $\IL^-(\J{4}_+)$, and that UFP is not the case for sublogics of $\IL^-(\J{1}, \J{5})$. 
We also show that $\ell$UFP holds for extensions of $\IL^-$. 
In Section \ref{Sec:CIP}, we prove that the logic $\IL^-(\J{2}_+, \J{5})$ has CIP by modifying a semantical proof of CIP for $\IL$ by Areces, Hoogland and de Jongh. 
We also notice that CIP for $\IL$ easily follows from CIP for $\IL^-(\J{2}_+, \J{5})$. 
In Section \ref{Sec:FPP}, we observe that FPP for $\IL^-(\J{2}_+, \J{5})$ immediately follows from our results in the previous sections. 
Also we give a syntactical proof of FPP for $\IL^-(\J{2}_+, \J{5})$. 
Moreover, we prove that $\IL^-(\J{4}, \J{5})$ has $\ell$FPP. 
In Section \ref{Sec:CE}, we provide counter models of $\ell$FPP for $\CL$ and $\IL^-(\J{1}, \J{5})$ and a counter model of FPP for $\IL^-(\J{1}, \J{4}_+, \J{5})$. 
As a consequence, we also show that CIP is not the case for these sublogics except for $\IL^-(\J{2}_+, \J{5})$ and $\IL$.

\section{Preliminaries}\label{Sec:Prelim}

\subsection{$\IL$ and its sublogics}

The interpretability logic $\IL$ is a base logic of modal logical investigations of the notion of relative interpretability (see \cite{Vis88,Vis90}). 
The language of $\IL$ consists of propositional variables $p, q, \ldots$, the propositional constant $\bot$, the logical connective $\to$, the unary modal operator $\Box$ and the binary modal operator $\rhd$. 
Other logical connectives, the propositional constant $\top$ and the modal operator $\Diamond$ are introduced as usual abbreviations. 
The formulas of $\IL$ are generated by the following grammar: 
\[
	A :: =  \bot \mid p \mid A \to A \mid \Box A \mid A \rhd A. 
\]
For each formula $A$, let $\boxdot A \equiv A \land \Box A$. 

\begin{defn}
The axioms of the modal propositional logic $\IL$ are as follows: 
\begin{description}
	\item [$\G{1}$] All tautologies in the language of $\IL$; 
	\item [$\G{2}$] $\Box(A \to B) \to (\Box A \to \Box B)$; 
	\item [$\G{3}$] $\Box(\Box A \to A) \to \Box A$; 
	\item [$\J{1}$] $\Box (A \to B) \to A \rhd B$; 
	\item [$\J{2}$] $(A \rhd B) \land (B \rhd C) \to A \rhd C$; 
	\item [$\J{3}$] $(A \rhd C) \land (B \rhd C) \to (A \lor B) \rhd C$; 
	\item [$\J{4}$] $A \rhd B \to (\Diamond A \to \Diamond B)$; 
	\item [$\J{5}$] $\Diamond A \rhd A$. 
\end{description}
The inference rules of $\IL$ are Modus Ponens $\dfrac{A\ \ \ A \to B}{B}$ and Necessitation $\dfrac{A}{\Box A}$. 
\end{defn}

The conservativity logic $\CL$ is obtained from $\IL$ by removing the axiom scheme $\J{5}$, that was introduced by Ignatiev \cite{Ign91} as a base logic of modal logical investigations of the notion of partial conservativity.
Several other sublogics of $\IL$ were introduced in \cite{KO20}. 
The basis for these newly introduced logics is the logic $\IL^-$. 

\begin{defn}
The language of $\IL^-$ is that of $\IL$, and the axioms of $\IL^-$ are $\G{1}, \G{2}, \G{3}, \J{3}$ and $\J{6}$: $\Box A \leftrightarrow (\neg A) \rhd \bot$. 
The inference rules of $\IL^-$ are Modus Ponens, Necessitation, $\R{1}$ $\dfrac{A \to B}{C \rhd A \to C \rhd B}$ and $\R{2}$ $\dfrac{A \to B}{B \rhd C \to A \rhd C}$. 
\end{defn}

For schemata $\Sigma_1, \ldots, \Sigma_n$, let $\IL^-(\Sigma_1, \ldots, \Sigma_n)$ be the logic obtained by adding $\Sigma_1, \ldots, \Sigma_n$ as axiom schemata to $\IL^-$. 
The following schemata $\J{2}_+$ and $\J{4}_+$ were introduced in \cite{KO20} and \cite{Vis88}, respectively:  

\begin{description}
	\item [$\J{2}_+$] $(A \rhd (B \lor C)) \land (B \rhd C) \to A \rhd C$; 
	\item [$\J{4}_+$] $\Box (A \to B) \to (C \rhd A \to C \rhd B)$. 
\end{description}

In this paper, we mainly deal with logics consisting of some of the axiom schemata $\J{1}, \J{2}_+, \J{4}_+$ and $\J{5}$ (see Figure \ref{Fig1} in Section \ref{Sec:Intro}). 
Then we have the following proposition. 

\begin{prop}\label{Prop:KO}
Let $A$, $B$ and $C$ be any formulas. 
\begin{enumerate}
	\item $\IL^- \vdash \Box \neg A \to A \rhd B$. 
	\item $\IL^- \vdash \Box(A \to B) \to (B \rhd C \to A \rhd C)$. 
	\item $\IL^- \vdash (\neg A \land B) \rhd C \to (A \rhd C \to B \rhd C)$. 
	\item $\IL^-(\J{4}_+) \vdash \J{4}$. 
	\item $\IL^-(\J{2}_+) \vdash \J{2} \land \J{4}_+$. 
	\item $\IL^-(\J{2}_+) \vdash (A \rhd B) \land ((B \land \neg C) \rhd C) \to (A \rhd C)$. 
	\item $\IL^-(\J{1}) \vdash A \rhd A$. 
	\item $\CL$ is deductively equivalent to $\IL^-(\J{1}, \J{2}_+)$. 
	\item $\IL$ is deductively equivalent to $\IL^-(\J{1}, \J{2}_+, \J{5})$. 
\end{enumerate}
\end{prop}
\begin{proof}
Except for 3, see \cite{KO20}. 
For 3, by $\J{3}$, $\IL^- \vdash ((\neg A \land B) \rhd C) \land (A \rhd C) \to ((\neg A \land B) \lor A) \rhd C$. 
Since $\IL^- \vdash B \to ((\neg A \land B) \lor A)$, we have $\IL^- \vdash ((\neg A \land B) \lor A) \rhd C \to B \rhd C$ by the rule $\R{2}$. 
Thus $\IL^- \vdash ((\neg A \land B) \rhd C) \land (A \rhd C) \to B \rhd C$. 
\end{proof}

The following lemma (Lemma \ref{l2}) plays an important role in our proofs of CIP and FPP for $\IL^-(\J{2}_+, \J{5})$ in Sections \ref{Sec:CIP} and \ref{Sec:FPP}.

\begin{fact}[See \cite{Vis97}]\label{l1}
For any formula $A$, 
\[
	\IL^- \vdash (A \lor \Diamond A) \leftrightarrow ((A \land \Box \lnot A) \lor \Diamond (A \land \Box \lnot A)). 
\]
\end{fact}

\begin{lem}\label{l2}
Let $A$ and $C$ be any formulas. 
\begin{enumerate}
	\item $\IL^-(\J{2}, \J{5}) \vdash ((A \land \Box \lnot A) \rhd C) \leftrightarrow (A \rhd C)$.
	\item $\IL^-(\J{2}_{+}, \J{5}) \vdash (C \rhd (A \land \Box \lnot A)) \leftrightarrow (C \rhd A)$.
\end{enumerate}
\end{lem}
\begin{proof}
In this proof, let $B \equiv (A \land \Box \neg A)$. 

1. $(\leftarrow)$: Since $\IL^- \vdash B \to A$, we have $\IL^- \vdash A \rhd C \to B \rhd C$ by $\R{2}$. 

$(\rightarrow)$:  
Since $\IL^-(\J{5}) \vdash \Diamond B \rhd B$, we have $\IL^-(\J{2}, \J{5}) \vdash B \rhd C \to \Diamond B \rhd C$. 
Hence, by $\J{3}$, 
\[
\IL^-(\J{2}, \J{5}) \vdash B \rhd C \to (B \lor \Diamond B)\rhd C. 
\]
By Fact \ref{l1} and $\R{2}$, we obtain
\[
\IL^-(\J{2}, \J{5}) \vdash  B \rhd C \to (A \lor \Diamond A) \rhd C.
\]
Since $\IL^- \vdash  A \to (A \lor \Diamond A)$, we obtain 
\[
\IL^-(\J{2}, \J{5}) \vdash B \rhd C \to A \rhd C 
\]
by $\R{2}$.

2. 
$(\rightarrow)$: This is immediate from $\IL^- \vdash B \to A$ and $\R{1}$. 

$(\leftarrow)$: 
Since $\IL^- \vdash A \to (A \lor \Diamond A)$, we obtain
\[
\IL^- \vdash C \rhd A \to C \rhd (A \lor \Diamond A)
\]
by $\R{1}$. 
Then, by Fact \ref{l1} and $\R{1}$, 
\[
\IL^- \vdash C \rhd A \to C \rhd (B \lor \Diamond B). 
\]
Since $\IL^-(\J{5}) \vdash \Diamond B \rhd B$, we obtain
\[
\IL^-(\J{2}_{+}, \J{5}) \vdash C \rhd A \to C \rhd B
\]
because $(C \rhd (\Diamond B \lor B)) \land (\Diamond B \rhd B) \to C \rhd B$ is an instance of $\J{2}_+$. 
\end{proof}

\subsection{$\IL^-$-frames and models}

\begin{defn}
We say that a system $\langle W, R, \{S_w\}_{w \in W} \rangle$ is an \textit{$\IL^-$-frame} if it satisfies the following three conditions: 
\begin{enumerate}
	\item $W$ is a non-empty set; 
	\item $R$ is a transitive and conversely well-founded binary relation on $W$; 
	\item For each $w \in W$, $S_w$ is a binary relation on $W$ with
\[
	\forall x, y \in W(x S_w y \Rightarrow w R x). 
\]
\end{enumerate}
A system $\langle W, R, \{S_w\}_{w \in W}, \Vdash \rangle$ is called an \textit{$\IL^-$-model} if $\langle W, R, \{S_w\}_{w \in W} \rangle$ is an $\IL^-$-frame and $\Vdash$ is a usual satisfaction relation on the Kripke frame $\langle W, R \rangle$ with the following additional condition: 
\[
w \Vdash A \rhd B \iff \forall x \in W(w R x \ \&\ x \Vdash A \Rightarrow \exists y \in W (x S_w y\ \&\ y \Vdash B)). 
\]
A formula $A$ is said to be \textit{valid} in an $\IL^-$-frame $\langle W, R, \{S_w\}_{w \in W} \rangle$ if for any satisfaction relation $\Vdash$ on the frame and any $w \in W$, $w\Vdash A$. 
\end{defn}

For each $w \in W$, let $\uparrow (w) : = \{x \in W : w R x\}$. 

\begin{prop}[See \cite{Vis88} and \cite{KO20}]\label{Prop:FC}
Let $\mathcal{F} = \langle W, R, \{S_w\}_{w \in W} \rangle$ be any $\IL^-$-frame. 
\begin{enumerate}
	\item $\J{1}$ is valid in $\mathcal{F}$ if and only if for any $w, x \in W$, if $w R x$, then $x S_w x$. 
	\item $\J{2}_+$ is valid in $\mathcal{F}$ if and only if $\J{4}_+$ is valid in $\mathcal{F}$ and for any $w \in W$, $S_w$ is transitive.  
	\item $\J{4}_+$ is valid in $\mathcal{F}$ if and only if for any $w \in W$, $S_w$ is a binary relation on $\uparrow (w)$. 
	\item $\J{5}$ is valid in $\mathcal{F}$ if and only if for any $w, x, y \in W$, $w R x$ and $x R y$ imply $x S_w y$. 
\end{enumerate}
\end{prop}

\begin{thm}[See \cite{KO20}, \cite{Ign91} and \cite{DeJVel90}]\label{Thm:KC}\leavevmode
Let $L$ be one of logics shown in Figure \ref{Fig1} in Section \ref{Sec:Intro}. 
Then for any formula $A$, the following are equivalent: 
\begin{enumerate}
	\item $L \vdash A$. 
	\item $A$ is valid in all (finite) $\IL^-$-frames in which all axioms of $L$ are valid. 
\end{enumerate}
\end{thm}

\subsection{The fixed point and the Craig interpolation properties}

For each formula $A$, let $v(A)$ be the set of all propositional variables contained in $A$. 

\begin{defn}
We say that a formula $A$ is \textit{modalized} in a propositional variable $p$ if every occurrence of $p$ in $A$ is in the scope of some modal operators $\Box$ or $\rhd$. 
\end{defn}

\begin{defn}
A logic $L$ is said to have the \textit{fixed point property} (FPP) if for any propositional variable $p$ and any formula $A(p)$ which is modalized in $p$, there exists a formula $F$ such that $v(F) \subseteq v(A) \setminus \{p\}$ and $L \vdash F \leftrightarrow A(F)$. 
\end{defn}

\begin{defn}
We say that the \textit{uniqueness of fixed points} (UFP) holds for a logic $L$ if for any propositional variables $p$, $q$ and any formula $A(p)$ which is modalized in $p$ and does not contain $q$, 
\[
	L \vdash \boxdot(p \leftrightarrow A(p)) \land \boxdot(q \leftrightarrow A(q)) \to (p \leftrightarrow q).
\]
\end{defn}

\begin{thm}[De Jongh and Visser \cite{DeJVis91}]\leavevmode
\begin{enumerate}
	\item $\IL$ has FPP. 
	\item UFP holds for $\IL$. 
\end{enumerate}
\end{thm}

In particular, de Jongh and Visser showed that a fixed point of a formula $A(p) \rhd B(p)$ is $A(\top) \rhd B(\Box \neg A(\top))$. 
Then a fixed point of every formula $A(p)$ which is modalized in $p$ is explicitly calculable by a usual argument. 

\begin{defn}
A logic $L$ is said to have the \textit{Craig interpolation property} (CIP) if for any formulas $A$ and $B$, there exists a formula $C$ such that $v(C) \subseteq v(A) \cap v(B)$, $L \vdash A \to C$ and $L \vdash C \to B$. 
\end{defn}

\begin{thm}[Areces, Hoogland and de Jongh \cite{AHD01}]\label{Thm:AHD}
$\IL$ has CIP. 
\end{thm}

\section{Uniqueness of fixed points}\label{Sec:UFP}

In this section, we investigate the uniqueness of fixed points for sublogics. 
First, we show that UFP holds for extensions of $\IL^-(\J{4}_{+})$. 
Secondly, we prove that UFP is not the case for sublogics of $\IL^-(\J{1}, \J{5})$. 
Then we investigate the newly introduced notion that a formula $A(p)$ is \textit{left-modalized} in a propositional variable $p$. 
We prove that UFP with respect to formulas which are left-modalized in $p$ ($\ell$UFP) holds for all extensions of $\IL^-$. 
At last, we discuss Smory\'nski's implication ``CIP + UFP $\Rightarrow$ FPP'' in our framework. 

\subsection{UFP}

By adapting Smory\'nski's argument \cite{Smo85}, de Jongh and Visser \cite{DeJVis91} showed that UFP holds for every logic closed under Modus Ponens and Necessitation, and containing $\G{1}$, $\G{2}$, $\G{3}$, $\mathbf{E1}$ and $\mathbf{E2}$, where
\begin{itemize}
	\item[$\mathbf{E1}$] $\Box(A \leftrightarrow B) \to (A \rhd C \leftrightarrow B \rhd C)$; 
	\item[$\mathbf{E2}$] $\Box(A \leftrightarrow B) \to (C \rhd A \leftrightarrow C \rhd B)$. 
\end{itemize}
%The inference rule of the $\mathbf{SR}_{0}$ are modus ponence, substitution and necessitation. 
%De Jongh and Visser proved that UFP holds for $\mathbf{SR}_{0}$. 

%\begin{fact}[\cite{DeJVis91}] 
%UFP holds for $\mathbf{SR}_{0}$. 
%\end{fact}

Since $\mathbf{E1}$ and $\mathbf{E2}$ are easy consequences of Proposition \ref{Prop:KO}.2 and $\J{4}_+$ respectively, we obtain the following theorem. 

\begin{thm}[UFP for $\IL^-(\J{4}_+)$]\label{UFP} 
UFP holds for every extension of the logic $\IL^-(\J{4}_+)$. 
\end{thm}

%\begin{proof}
%It suffices to prove that UFP holds for $\IL^{-}(\J{4}_{+})$. 
%By Proposition \ref{Prop:KO} and $\J{4}_{+}$, we obtain that $\IL^-(\J{4}_{+})$ proves $\mathbf{E1}$ and $\mathbf{E2}$. 
%Hence, $\IL^-(\J{4}_{+})$ is extension of $\mathbf{SR}_{0}$. 
%Therefore, UFP holds for $\IL^{-}(\J{4}_{+})$. 
%\end{proof}

As shown in \cite{DeJVis91}, in the proof of Theorem \ref{UFP}, the use of the following substitution principle is essential. 

\begin{prop}[The Substitution Principle]\label{p2}
Let $A$, $B$ and $C(p)$ be any formulas.  
\begin{enumerate}
	\item $\IL^-(\J{4}_{+}) \vdash \boxdot(A \leftrightarrow B) \to (C(A) \leftrightarrow C(B))$.
	\item If $C(p)$ is modalized in $p$, then $\IL^-(\J{4}_{+}) \vdash \Box (A \leftrightarrow B) \to (C(A) \leftrightarrow C(B))$.
\end{enumerate} 
\end{prop}

Proposition \ref{p2}.2 shows that every extension $L$ of $\IL^-(\J{4}_+)$ proves $\Box(A \leftrightarrow B) \to (C(A) \leftrightarrow C(B))$ for any formula $C(p)$ which is modalized in $p$. 
We notice that the converse of this statement also holds. 

\begin{prop}
Let $L$ be any extension of $\IL^-$. 
Suppose that for any formula $C(p)$ which is modalized in $p$, $L \vdash \Box(A \leftrightarrow B) \to (C(A) \leftrightarrow C(B))$. 
Then $L \vdash \J{4}_+$.
\end{prop}
\begin{proof}
Let $A$, $B$ and $C$ be any formulas and assume $p \notin v(C)$. 
Then the formula $C \rhd p$ is modalized in $p$. 
By the supposition, we have 
\[
	L \vdash \Box(A \leftrightarrow A \land B) \to (C \rhd A \leftrightarrow C \rhd (A \land B)).
\]
Since $\IL^- \vdash \Box(A \to B) \to \Box(A \leftrightarrow A \land B)$ and $\IL^- \vdash C \rhd (A \land B) \to C \rhd B$, we obtain $L \vdash \Box(A \to B) \to (C \rhd A \to C \rhd B)$. 
\end{proof}

On the other hand, we show that UFP does not hold for sublogics of $\IL^{-}(\J{1}, \J{5})$ in general.

\begin{prop}
Let $p, q$ be distinct propositional variables. 
Then, 
\[
\IL^{-}(\J{1}, \J{5}) \nvdash \boxdot(p \leftrightarrow (\top \rhd \lnot p)) \land \boxdot(q \leftrightarrow (\top \rhd \lnot q)) \to (p \leftrightarrow q).
\]
\end{prop}

\begin{proof}
We define an $\IL^-$-frame $\mathcal{F} = \seq{W, R, \{S_{w}\}_{w \in W}}$ as follows: 
\begin{itemize}
	\item $W := \{w, x, y\}$; 
	\item $R := \{\seq{w, x}\}$; 
	\item $S_{w}:=\{\seq{x, x}, \seq{x, y}\}$, $S_{x} := \emptyset$, $S_{y} := \emptyset$. 
\end{itemize}

Obviously, by Proposition \ref{Prop:FC}, $\IL^-(\J{1}, \J{5})$ is valid in $\mathcal{F}$. 
Let $\Vdash$ be a satisfaction relation on $\mathcal{F}$ satisfying the following conditions: 
%the following properties hold by the definition of $R$.

%\begin{enumerate}
%	\item For any $w, x \in W$, $wRx \Rightarrow x S_{w} x$.
%	\item For any $w, x, y \in W$, $w R x R y \Rightarrow xS_{w} y$.
%\end{enumerate}

%Therefore, by , axioms of $\IL^-(\J{1}, \J{5})$ are valid in $\mathcal{F}$. 
%Also, we define $\mathbf{IL}^{-}$-model $\seq{W, R, \{S_{z}\}_{z \in W}, \Vdash}$ as follows:
\begin{itemize}
	\item $w \Vdash p$ and $w \nVdash q$;
	\item $x \Vdash p$ and $x \Vdash q$;
	\item $y \nVdash p$ and $y \Vdash q$. 
\end{itemize}

\begin{figure}[th]
\centering
\begin{tikzpicture}
\node [draw, circle] (w) at (0,0) {$w$};
\node [draw, circle] (x) at (0,1.5) {$x$};
\node [draw, circle] (y) at (1.5,1.5) {$y$};

\draw [thick, ->] (w)--(x);
\draw [thick, ->, dashed] (x)--(y);
\draw [thick, ->, dashed] (0.2,1.7) arc (-45:225:0.3);

\draw (-0.3, 0) node[left] {$p$};
\draw (-0.3, 1.5) node[left] {$p, q$};
\draw (1.8, 1.5) node[right] {$q$};

\end{tikzpicture}
\caption{A counter model of UFP for $\IL^-(\J{1}, \J{5})$}
\end{figure}

We prove $w \Vdash \boxdot(p \leftrightarrow (\top \rhd \lnot p)) \land \boxdot(q \leftrightarrow (\top \rhd \lnot q)) \land \lnot (p \leftrightarrow q)$. 
Since $w \Vdash p$ and $w \nVdash q$, $w \Vdash \lnot (p \leftrightarrow q)$ is obvious. 
We show $w \Vdash (p \leftrightarrow (\top \rhd \lnot p)) \land (q \leftrightarrow (\top \rhd \lnot q))$. 
Since $w \Vdash p$ and $w \nVdash q$, it suffices to prove $w \Vdash \top \rhd \lnot p$ and $w \Vdash \lnot(\top \rhd \lnot q)$. 

\vspace{2mm}
\noindent
$w \Vdash \top \rhd \lnot p$: Let $z \in W$ be any element with $w R z$. 
Then $z = x$. 
Since $x S_{w} y$ and $y \Vdash \lnot p$, we obtain $w \Vdash \top \rhd \lnot p$. 

\vspace{2mm}
\noindent
$w \Vdash \lnot(\top \rhd \lnot q)$: 
Let $z \in W$ be any element with $x S_{w} z$. 
Then $z = x$ or $z = y$. 
In either case, we obtain $z \Vdash q$. 
Since $x R y$, we conclude $w \Vdash \lnot(\top \rhd \lnot q)$. 

At last, we show $w \Vdash \Box(p \leftrightarrow (\top \rhd \lnot p)) \land \Box(q \leftrightarrow (\top \rhd \lnot q))$. 
Let $z \in W$ be such that $w R z$. 
Then $z = x$. 
Since there is no $z' \in W$ such that $x R z'$, $x \Vdash (\top \rhd \lnot p) \land (\top \rhd \lnot q)$. 
Since $x \Vdash p$ and $x \Vdash q$, we have $x \Vdash (p \leftrightarrow (\top \rhd \lnot p)) \land (q \leftrightarrow (\top \rhd \lnot q))$. 
Hence, we obtain $w \Vdash \Box(p \leftrightarrow (\top \rhd \lnot p)) \land \Box(q \leftrightarrow (\top \rhd \lnot q))$.

Therefore, $w \Vdash \boxdot(p \leftrightarrow (\top \rhd \lnot p)) \land \boxdot(q \leftrightarrow (\top \rhd \lnot q)) \land \lnot (p \leftrightarrow q)$. 

\end{proof}

\subsection{$\ell$UFP}

Even for extensions of $\IL^-$, Proposition \ref{Prop:KO}.2 suggests that the uniqueness of fixed points may hold with respect to formulas in some particular forms. 
From this perspective, we introduce the notion that formulas are left-modalized in $p$.

\begin{defn}
We say that a formula $A$ is \textit{left-modalized} in a propositional variable $p$ if $A$ is modalized in $p$ and for any subformula $B \rhd C$ of $A$, $p \notin v(C)$. 
\end{defn}

Then we obtain the following version of the substitution principle. 

\begin{prop}\label{p'2}
Let $A$, $B$ and $C(p)$ be any formulas such that for any subformula $D \rhd E$ of $C$, $p \notin v(E)$. 
\begin{enumerate}
	\item $\IL^- \vdash \boxdot(A \leftrightarrow B) \to (C(A) \leftrightarrow C(B))$.
	\item If $C(p)$ is left-modalized in $p$, then $\IL^- \vdash \Box(A \leftrightarrow B) \to (C(A) \leftrightarrow C(B))$.
\end{enumerate}
\end{prop}
\begin{proof}
1. This is proved by induction on the construction of $C(p)$. 
We only prove the case $C(p) \equiv D(p) \rhd E$ (By our supposition, $p \notin v(E)$). 
For any subformula $D' \rhd E'$ of $D$, it is also a subformula of $C$, and hence $p \notin v(E')$. 
Then, by induction hypothesis, we obtain
\[
\IL^- \vdash \boxdot(A \leftrightarrow B) \to (D(A) \leftrightarrow D(B)). 
\]
Then, $\IL^- \vdash \Box(A \leftrightarrow B) \to \Box(D(A) \leftrightarrow D(B))$. Therefore, by Proposition \ref{Prop:KO}.2, 
\[
\IL^- \vdash \Box(A \leftrightarrow B) \to (D(A) \rhd E \leftrightarrow D(B) \rhd E). 
\]
Since $p \notin v(E)$, $C(A) \equiv (D(A) \rhd E)$ and $C(B) \equiv (D(B) \rhd E)$. Therefore, 
\[
\IL^- \vdash \Box(A \leftrightarrow B) \to (C(A) \leftrightarrow C(B)). 
\]

2. This follows from our proof of 1. 
\end{proof}

%\begin{prop}\label{p'3} Let $A$, $B$, $\Box C(p)$ and $D(p) \rhd E$ be modal formulas. 
%Suppose $p$ is left-modalized in $\Box C(p)$ and $D(p) \rhd E$. Then, 
%\begin{enumerate}
%\item $\IL^- \vdash \Box(A \leftrightarrow B) \to (\Box C(A) \leftrightarrow \Box C(B))$
%\item $\IL^- \vdash \Box(A \leftrightarrow B) \to (D(A) \rhd E \leftrightarrow D(B) \rhd E)$
%\end{enumerate}
%\end{prop}

%\begin{proof}
%1. This is proved by Proposition \ref{p2} and necessitation. 2: We obtain by proof of Proposition \ref{p2}. 
%\end{proof}

We introduce our restricted versions of UFP and FPP. 

\begin{defn}
We say that $\ell$UFP holds for a logic $L$ if for any formula $A(p)$ which is left-modalized in $p$, $L \vdash \boxdot(p \leftrightarrow A(p)) \land \boxdot(q \leftrightarrow A(q)) \to (p \leftrightarrow q)$.
\end{defn}

\begin{defn}
We say that a logic $L$ has $\ell$FPP if for any formula $A(p)$ which is left-modalized in $p$, there exists a formula $F$ such that $v(F) \subseteq v(A) \setminus \{p\}$ and $L \vdash F \leftrightarrow A(F)$. 
\end{defn}

Then $\ell$UFP holds for every our sublogic of $\IL$.

%\begin{prop}\label{pp1}
%Let $A(p)$ be a modal formula in which $p$ is left-modalized. 
%Then, there exist a propositional formula $B(p_{1},\ldots, p_{n+m})$ and modal formulas $C_{1}(p)$, $\ldots$, $C_{n}(p)$,  $D_{1}(p)$, $\ldots$, $D_{m}(p)$, $E_{1}$, $\ldots$, $E_{m}$ such that $A \equiv B(\Box C_{1}(p),\ldots, \Box C_{n}(r),(D_{1}(p) \rhd E_{1}),\ldots, (D_{m}(p) \rhd E_{m}))$ and $p \notin v(E_{j})$for any $j$($1 \leq j \leq m$). 
%\end{prop}

%\begin{proof}
%This is proved by induction on the construction of $A(p)$. 
%\end{proof}

\begin{thm}[$\ell$UFP for $\IL^-$]\label{UFP'} 
$\ell$UFP holds for all extensions of $\IL^-$. 
\end{thm}
\begin{proof}
Let $A(p)$ be any formula which is left-modalized in $p$. 
Then by Proposition \ref{p'2}.2, $\IL^- \vdash \Box(p \leftrightarrow q) \to (A(p) \leftrightarrow A(q))$. 
%By Proposition \ref{p'1}, there exist a propositional formula $B$ and modal formulas $\Box C_{1}(p)$, $\ldots$, $\Box C_{n}(p)$, $D_{1}(p) \rhd E_{1},\ldots, D_{m}(p) \rhd E_{m}$ such that $A(r) \equiv B(\Box C_{1}(p),\ldots, \Box C_{n}(p),D_{1}(p) \rhd E_{1},\ldots, D_{m}(p) \rhd E_{m})$. Since $p$ is left-modalized in $A(p)$, for any $i$ which satisfies $1 \leq i \leq n$, $p$ is left-modalized in $\Box C_{i}(p)$. 
%Therefore, we obtain 
%\[
%\IL^- \vdash \Box(p \leftrightarrow q) \to (\Box C_{i}(p) \leftrightarrow \Box C_{i}(q)) 
%\]
%by Proposition \ref{p'3}.1. 
%Similarly, for any $j$ which satisfies $1 \leq j \leq m$, we obtain 
%\[
%\IL^- \vdash \Box(p \leftrightarrow q) \to (D_{j}(p) \rhd E_{j} \leftrightarrow D_{j}(q) \rhd E_{j}) 
%\]
%by Proposition \ref{p'3}.2. 
%Since $B$ is propositional formula, we obtain 
Therefore, 
\begin{align*}
\IL^- \vdash \boxdot(p \leftrightarrow A(p)) \land \boxdot(q \leftrightarrow A(q)) 
 & \to (\Box(p \leftrightarrow q) \to (A(p) \leftrightarrow A(q))) \\
 & \to (\Box(p \leftrightarrow q) \to (p \leftrightarrow q)) \\
 & \to (\Box(\Box(p \leftrightarrow q) \to (p \leftrightarrow q))) \\
 & \to \Box(p \leftrightarrow q)\\
 & \to (p \leftrightarrow q). 
\end{align*}
\end{proof}

\subsection{Applications of Smory\'nski's argument}

We have shown that UFP and the substitution principle hold for extensions of $\IL^-(\J{4}_+)$ (Theorem \ref{UFP} and Proposition \ref{p2}). 
Then by applying Smory\'nski's argument \cite{Smo78}, we prove that for any appropriate extension of $\IL^-(\J{4}_+)$, CIP implies FPP. 

\begin{lem}\label{FPP-CU}
Let $L$ be any extension of $\IL^{-}(\J{4}_{+})$ that is closed under substituting a formula for a propositional variable. 
If $L$ has CIP, then $L$ also has FPP.  
\end{lem}
\begin{proof}
Suppose $L \supseteq \IL^-(\J{4}_+)$ and $L$ has CIP. 
Let $A(p)$ be any formula modalized in $p$. 
Then by Theorem \ref{UFP}, 
\[
	L \vdash \boxdot(p \leftrightarrow A(p)) \land \boxdot(q \leftrightarrow A(q)) \to (p \leftrightarrow q).
\]
We have
\[
	L \vdash \boxdot(p \leftrightarrow A(p)) \land p \to (\boxdot(q \leftrightarrow A(q))  \to q).
\]
Since $L$ has CIP, there exists a formula $F$ such that $v(F) \subseteq v(A) \setminus \{p\}$, $L \vdash \boxdot(p \leftrightarrow A(p)) \land p \to F$ and $L \vdash F \to (\boxdot(q \leftrightarrow A(q)) \to q)$. 
Since $q \notin v(F)$, we have $L \vdash F \to (\boxdot(p \leftrightarrow A(p)) \to p)$ by substituting $p$ for $q$. 
Then 
\[
	L \vdash \boxdot(p \leftrightarrow A(p)) \to (F \leftrightarrow p). 
\]
By substituting $A(F)$ for $p$, we get
\begin{eqnarray}\label{FP1}
	L \vdash \boxdot(A(F) \leftrightarrow A(A(F))) \to (F \leftrightarrow A(F)). 
\end{eqnarray}
Then
\[
	L \vdash \Box(A(F) \leftrightarrow A(A(F))) \to \Box (F \leftrightarrow A(F)). 
\]
Since $A(p)$ is modalized in $p$, by Proposition \ref{p2}.2, 
\[
	L \vdash \Box(A(F) \leftrightarrow A(A(F))) \to (A(F) \leftrightarrow A(A(F))). 
\]
Then by applying the axiom scheme $\G{3}$, we obtain $L \vdash A(F) \leftrightarrow A(A(F))$. 
From this with (\ref{FP1}), we conclude $L \vdash F \leftrightarrow A(F)$. 
Therefore $F$ is a fixed point of $A(p)$ in $L$. 
\end{proof}

Also we have shown that $\ell$UFP and the substitution principle with respect to left-modalized formulas hold for extensions of $\IL^-$ (Theorem \ref{UFP'} and Proposition \ref{p'2}). 
Thus our proof of Lemma \ref{FPP-CU} also works for the following lemma.

\begin{lem}\label{lFPP-ClU}
Let $L$ be any extension of $\IL^{-}$ that is closed under substituting a formula for a propositional variable. 
If $L$ has CIP, then $L$ also has $\ell$FPP.  
\end{lem}

\section{The Craig interpolation property}\label{Sec:CIP}

In this section, we prove the following theorem. 

\begin{thm}[CIP for $\IL^-(\J{2}_+, \J{5})$]\label{T3-1}
The logic $\IL^-(\J{2}_+, \J{5})$ has CIP. 
%Suppose that $\vdash A_0 \to B_0$. Then there exists $I \in \mathcal{L}_{A_0} \cap \mathcal{L}_{B_0}$ such that:
%\begin{equation*}
%\vdash A_0 \to I \text{ and } \vdash B_0 \to {\sim} I.
%\end{equation*}
\end{thm}

Our proof of Theorem \ref{T3-1} is based on a semantical proof of CIP for $\IL$ due to Areces, Hoogland and de Jongh \cite{AHD01}.

\subsection{Preparations for our proof of Theorem \ref{T3-1}}

In this subsection, we prepare several definitions and prove some lemmas that are used in our proof of Theorem \ref{T3-1}. 
Only in this section, we write $\vdash A$ instead of $\IL^-(\J{2}_+, \J{5}) \vdash A$ if there is no confusion. 
Notice that by Proposition \ref{Prop:KO}, $\vdash \J{2} \land \J{4} \land \J{4}_+$. 

For a formula $A$, we define the formula ${\sim} A$ as follows:
\begin{eqnarray*}
{\sim} A :\equiv \left\{
\begin{array}{ll}
B & \text{if}\ A \equiv \neg B\ \text{for some formula}\ B, \\
\neg A & \text{otherwise.}
\end{array}
\right.
\end{eqnarray*}

For a set $X$ of formulas, by $\mathcal{L}_X$ we denote the set of all formulas built up from $\bot$ and propositional variables occurring in formulas in $X$. 
We simply write $\mathcal{L}_A$ instead of $\mathcal{L}_{\{ A \}}$. 
For a finite set $X$ of formulas, let $\bigwedge X$ be a conjunction of all elements of $X$. 
For the sake of simplicity, only in this section, $\vdash \bigwedge X \to A$ will be written as $\vdash X \to A$.

For a set $\Phi$ of formulas, we define 
\[
	\Phi_\rhd := \{A : \text{there exists a formula}\ B \ \text{such that}\ A \rhd B \in \Phi\ \text{or}\ B \rhd A \in \Phi \}.
\]
\begin{defn}
A set $\Phi$ of formulas is said to be \textit{adequate} if it satisfies the following conditions:
\begin{enumerate}
\item $\Phi$ is closed under taking subformulas and the ${\sim}$-operation;
\item $\bot \in \Phi_\rhd$;
\item If $A, B \in \Phi_\rhd$, then $A \rhd B \in \Phi$;
\item If $A \in \Phi_\rhd$, then $\Box {\sim} A \in \Phi$.
\end{enumerate}
\end{defn}

Note that for any finite set $X$ of formulas, there exists the smallest finite adequate set $\Phi$ containing $X$. 
We denote this set by $\Phi_X$.

\begin{defn}\leavevmode
\begin{enumerate}
\item A pair $(\Gamma_1, \Gamma_2)$ of finite sets of formulas is said to be \textit{separable} if for some formula $I \in \mathcal{L}_{\Gamma_1} \cap \mathcal{L}_{\Gamma_2}$, $\vdash \Gamma_1 \to I$ and $\vdash \Gamma_2 \to \neg I$. 
A pair is said to be \textit{inseparable} if it is not separable.
\item A pair $(\Gamma_1, \Gamma_2)$ of finite sets of formulas is said to be \textit{complete} if it is inseparable and
\begin{itemize}
\item For each $F \in \Phi_{\Gamma_1}$, either $F \in \Gamma_1$ or ${\sim} F \in \Gamma_1$;
\item For each $F \in \Phi_{\Gamma_2}$, either $F \in \Gamma_2$ or ${\sim} F \in \Gamma_2$. 
\end{itemize}
\end{enumerate}
\end{defn}

We say a finite set $X$ of formulas is \textit{consistent} if $\nvdash X \to \bot$. 
If a pair $(\Gamma_1, \Gamma_2)$ is inseparable, then it can be shown that both of $\Gamma_1$ and $\Gamma_2$ are consistent. 

In the rest of this subsection, we fix some sets $X$ and $Y$ of formulas. Put $\Phi^1 := \Phi_X$ (resp.~$\Phi^2 := \Phi_Y$) and $\mathcal{L}_1 := \mathcal{L}_X$ (resp.~$\mathcal{L}_2 := \mathcal{L}_Y$). 
Let $X' \subseteq \Phi^1$ and $Y' \subseteq \Phi^2$. 
It is easily proved that if $(X', Y')$ is inseparable, then for any formula $A \in \Phi^1$, at least one of $(X' \cup \{A\}, Y')$ and $(X' \cup \{{\sim} A\}, Y')$ is inseparable. 
Also a similar statement holds for $\Phi^2$ and $Y'$. 
Then we obtain the following proposition. 

\begin{prop}
If $(X, Y)$ is inseparable, then there exists some complete pair $\Gamma' = (\Gamma_1, \Gamma_2)$ such that $X \subseteq \Gamma_1 \subseteq \Phi^1$ and $Y \subseteq \Gamma_2 \subseteq \Phi^2$.
\end{prop}

Let $K(\Phi^1, \Phi^2)$ be the set of all complete pairs $(\Gamma_1, \Gamma_2)$ satisfying $\Gamma_1 \subseteq \Phi^1$ and $\Gamma_2 \subseteq \Phi^2$. 
Note that the set $K(\Phi^1, \Phi^2)$ is finite. 
For each $\Gamma \in K(\Phi^1, \Phi^2)$, let $\Gamma_1$ and $\Gamma_2$ be the first and the second components of $\Gamma$, respectively. 

%\prec relation
\begin{defn}
We define a binary relation $\prec$ on $K(\Phi^1, \Phi^2)$ as follows:
For $\Gamma, \Delta \in K(\Phi^1, \Phi^2)$, 
\begin{eqnarray*}
\Gamma \prec \Delta :\Leftrightarrow
\begin{array}{l}
\text{For}\ i = \{ 1, 2 \},\ \text{if}\ \Box A \in \Gamma_i, \ \text{then}\  \Box A, A \in \Delta_i, \ \text{and} \\
\text{there exists some}\ \Box B\ \text{such that}\  \Box B \in \Delta_1 \cup \Delta_2 \ \text{and} \ \Box B \not \in \Gamma_1 \cup \Gamma_2.
\end{array}
\end{eqnarray*}
\end{defn}

Then $\prec$ is a transitive and conversely well-founded binary relation on $K(\Phi^1, \Phi^2)$. 

%A-critical successor
\begin{defn}\label{D3-1}
Let $\Gamma, \Delta \in K(\Phi^1, \Phi^2)$ and $A \in \Phi^1_\rhd \cup \Phi^2_\rhd$. 
We say that $\Delta$ is an \textit{$A$-critical successor} of $\Gamma$ (write $\Gamma \prec_A \Delta$) if the following conditions are met:
\begin{enumerate}
\item $\Gamma \prec \Delta$;
\item If $A \in \Phi^1_\rhd$, then
	\begin{align*}
	\Gamma_1^A:= & \{ \Box {\sim} B, {\sim} B : B \rhd A \in \Gamma_1 \} \subseteq \Delta_1; \\
	\Gamma_2^A:= & \{ \Box {\sim} C, {\sim} C : C \in \Phi^2_\rhd \ \text{and for some}\ I \in \mathcal{L}_1 \cap \mathcal{L}_2, \\
	& \hspace{6em}  \vdash \Gamma_1 \to (I \land \neg A) \rhd A \ \&\ \vdash \Gamma_2 \to C \rhd I \} \subseteq \Delta_2.
	\end{align*}
\item If $A \in \Phi^2_\rhd$, then 
	\begin{align*}
	\Gamma_1^A := & \{ \Box {\sim} B, {\sim} B : B \in \Phi^1_\rhd\ \text{and for some}\ I \in \mathcal{L}_1 \cap \mathcal{L}_2, \\
	& \hspace{6em}  \vdash \Gamma_1 \to B \rhd I\ \&\ \vdash \Gamma_2 \to (I \land \neg A) \rhd A \} \subseteq \Delta_1; \\
	\Gamma_2^A := & \{ \Box {\sim} C, {\sim} C : C \rhd A \in \Gamma_2 \} \subseteq \Delta_2.
	\end{align*}
\end{enumerate}
\end{defn}

From the following claim, Definition \ref{D3-1} makes sense. 

\begin{cl}\label{C3-1}
If $A \in \Phi^1_\rhd \cap \Phi^2_\rhd$, then the sets $\Gamma_1^A$ in clauses 2 and 3 of Definition \ref{D3-1} coincide. 
This is also the case for $\Gamma_2^A$. 
\end{cl}
\begin{proof}
We prove only for $\Gamma_1^A$. 
It suffices to show that for any formula $B$, the following are equivalent: 
\begin{enumerate}
	\item $B \rhd A \in \Gamma_1$. 
	\item $B \in \Phi^1_\rhd$ and for some $I \in \mathcal{L}_1 \cap \mathcal{L}_2$, $\vdash \Gamma_1 \to B \rhd I$ and $\vdash \Gamma_2 \to (I \land \neg A) \rhd A$. 
\end{enumerate}

$(1 \Rightarrow 2)$: 
Suppose $B \rhd A \in \Gamma_1$, then $B \in \Phi^1_\rhd$. 
By Proposition \ref{Prop:KO}.1, we have $\IL^- \vdash (A \land \neg A) \rhd A$ because $\IL^- \vdash \Box \neg(A \land \neg A)$. 
Since $A \in \mathcal{L}_1 \cap \mathcal{L}_2$, the clause 2 holds by letting $I \equiv A$. 

$(2 \Rightarrow 1)$: Assume that the clause 2 holds. 
Then $A \rhd B \in \Phi^1$ because $A, B \in \Phi^1_\rhd$. 
Suppose, towards a contradiction, that $\neg (B \rhd A) \in \Gamma_1$. 
By Proposition \ref{Prop:KO}.6, $\vdash (B \rhd I) \land ((I \land \neg A) \rhd A) \to B \rhd A$. 
Then we obtain $\vdash \Gamma_1 \to \neg((I \land \neg A) \rhd A)$. 
This contradicts the inseparability of $\Gamma$ because $(I \land \neg A) \rhd A \in \mathcal{L}_1 \cap \mathcal{L}_2$. 
Hence $\neg (B \rhd A) \notin \Gamma_1$. 
Since $\Gamma$ is complete, $B \rhd A \in \Gamma_1$. 
\end{proof}

\begin{lem}\label{L3-1}
For $\Gamma, \Delta \in K(\Phi^1, \Phi^2)$, if $\Gamma \prec \Delta$, then $\Gamma \prec_\bot \Delta$.
\end{lem}
\begin{proof}
Notice that $\bot \in \Phi^1_\rhd \cap \Phi^2_\rhd$. 
By Claim \ref{C3-1}, it suffices to show that if $C \rhd \bot \in \Gamma_1$ (resp.~$\Gamma_2$) then $\Box {\sim} C, {\sim} C \in \Delta_1$ (resp.~$\Delta_2$). 
Suppose $C \rhd \bot \in \Gamma_1$. 
Then by $(\J{6})$, $\vdash \Gamma_1 \to \Box {\sim} C$. 
Note that $\Box {\sim} C \in \Phi^1$, and hence $\Box {\sim} C \in \Gamma_1$. 
By $\Gamma \prec \Delta$, $\Box {\sim} C, {\sim} C \in \Delta_1$. 
The case $C \rhd \bot \in \Gamma_2$ is proved similarly. 
Therefore $\Gamma \prec_\bot \Delta$.
\end{proof}

\begin{lem}\label{L3-2}
For $\Gamma, \Delta, \Theta \in K(\Phi^1, \Phi^2)$ and $A \in \Phi^1_\rhd \cup \Phi^2_\rhd$, if $\Gamma \prec_A \Delta$ and $\Delta \prec \Theta$, then $\Gamma \prec_A \Theta$.
\end{lem}
\begin{proof}
We only prove the case $A \in \Phi^1_\rhd$. 
Let $\Gamma_1^A$ and $\Gamma_2^A$ be the sets as in Definition \ref{D3-1}. 
If $\Box {\sim} B, {\sim} B \in \Gamma_1^A$, then $\Box {\sim} B \in \Delta_1$ because $\Gamma \prec_A \Delta$. 
Thus $\Box {\sim} B, {\sim} B \in \Theta_1$ because $\Delta \prec \Theta$. 
Similarly, if $\Box {\sim} C, {\sim} C \in \Gamma_2^A$, then $\Theta_2$ contains $\Box {\sim} C$ and ${\sim} C$. 
This means $\Gamma \prec_A \Theta$.
\end{proof}

%Main Lemmas
In order to prove the Truth Lemma (Lemma \ref{L3-3}), we show the following two lemmas. 
%In the proof, we use a fact which will be proved in Section 5 (See Lemma \ref{l2}).

\begin{lem}\label{L3-m1}
Let $\Gamma \in K(\Phi^1, \Phi^2)$. 
If $\neg (G \rhd F) \in \Gamma_1 \cup \Gamma_2$, then there exists a pair $\Delta \in K(\Phi^1, \Phi^2)$ such that
\begin{enumerate}
\item $\Gamma \prec_F \Delta$;
\item $G, \Box {\sim} F \in \Delta_1 \cup \Delta_2$.
\end{enumerate}
\end{lem}
\begin{proof}
Suppose $\neg (G \rhd F) \in \Gamma_1$. 
Let
\begin{align*}
	X':= & \boxdot \Gamma_1 \cup \{ G, \Box {\sim} G, \Box {\sim} F \} \cup \{ \Box {\sim} A, {\sim} A : A \rhd F \in \Gamma_1 \}; \\
	Y':= & \boxdot \Gamma_2 \cup \{ \Box {\sim} B, {\sim} B : B \in \Phi^2_\rhd \ \text{and for some}\  I \in \mathcal{L}_1 \cap \mathcal{L}_2, \\
	& \hspace{10em}  \vdash \Gamma_1 \to (I \land \neg F) \rhd F\ \&\ \vdash \Gamma_2 \to B \rhd I \},
	\end{align*}
where $\boxdot \Gamma_i$ ($i =1,2$) denotes the set $\{ \Box C, C : \Box C \in \Gamma_i \}$.

We claim $\Box {\sim} G \not \in \Gamma_1 \cup \Gamma_2$. 
Assume $\Box \neg G \in \Gamma_1$. 
Then $\vdash \Gamma_1 \to \Box \neg G$. 
By Proposition \ref{Prop:KO}.1, $\vdash \Box \neg G \to G \rhd F$. 
Hence $\vdash \Gamma_1 \to G \rhd F$. 
This implies that $\Gamma_1$ is inconsistent, a contradiction. 
Thus $\Box {\sim} G \not \in \Gamma_1$. 
Moreover, if $\Box {\sim} G \in \Gamma_2$, then $\Diamond G$ separates $(\Gamma_1, \Gamma_2)$ because $\Diamond G \in \mathcal{L}_1 \cap \mathcal{L}_2$. 
This contradicts the inseparability of $\Gamma$. 
Hence $\Box {\sim} G \not \in \Gamma_2$. 

We show that $( X', Y' )$ is inseparable. 
Suppose, for a contradiction, that $J \in \mathcal{L}_1 \cap \mathcal{L}_2$ separates $(X', Y')$. 
From $\vdash Y' \to \neg J$,
\begin{equation*}
\vdash \boxdot \Gamma_2 \to \left( J \to \bigvee_{j \in \kappa}{(\Diamond B_j \lor B_j)} \right),
\end{equation*}
where $\kappa$ is an appropriate index set for $Y'$. 
Then for each $j \in \kappa$, $B_j \in \Phi^2_\rhd$ and there exists a formula $I_j \in \mathcal{L}_1 \cap \mathcal{L}_2$ such that
\begin{equation}\label{Fml3}
\vdash \Gamma_1 \to (I_j \land \neg F) \rhd F, \ \text{and}
\end{equation}
\begin{equation}\label{Fml4}
\vdash \Gamma_2 \to B_j \rhd I_j.
\end{equation}
Then
\[
	\vdash \Gamma_2 \to \Box \left( J \to \bigvee_{j \in \kappa}{(\Diamond B_j \lor B_j)} \right). 
\]
By Proposition \ref{Prop:KO}.2, 
\[
	\vdash \Gamma_2 \to \left( \left( \bigvee_{j \in \kappa}{(\Diamond B_j \lor B_j)} \right) \rhd \bigvee_{j \in \kappa} I_j \to J \rhd \bigvee_{j \in \kappa} I_j \right). 
\]
By (\ref{Fml4}), $\J{2}$, $\J{3}$ and $\J{5}$, we have $\displaystyle \vdash \Gamma_2 \to \left(\bigvee_{j \in \kappa}{(\Diamond B_j \lor B_j)}\right) \rhd \bigvee_{j \in \kappa} I_j$. 
Hence
\begin{equation}\label{Fml5}
\vdash \Gamma_2 \to J \rhd \bigvee_{j \in \kappa}I_j.
\end{equation}

On the other hand, from $\vdash X' \to J$,
\begin{align*}
 &\vdash \boxdot \Gamma_1 \to \left( \neg J \land G \land \Box \neg G \to \bigvee_{A \rhd F \in \Gamma_1}{(\Diamond A \lor A)} \lor \Diamond F \right),\\
 &\vdash \Gamma_1 \to \Box \left( \neg J \land G \land \Box \neg G \to \bigvee_{A \rhd F \in \Gamma_1}{(\Diamond A \lor A)} \lor \Diamond F \right), \\
 &\vdash \Gamma_1 \to \left(\left( \bigvee_{A \rhd F \in \Gamma_1}{(\Diamond A \lor A)} \lor \Diamond F\right) \rhd F \to
	(\neg J \land G \land \Box \neg G) \rhd F \right). \tag{By Proposition \ref{Prop:KO}.2}
\end{align*}

By $\J{2}$, $\J{3}$ and $\J{5}$, we have $\displaystyle \vdash \Gamma_1 \to \left( \bigvee_{A \rhd F \in \Gamma_1}{(\Diamond A \lor A)} \lor \Diamond F\right) \rhd F$. 
Hence we obtain $\vdash \Gamma_1 \to (\neg J \land G \land \Box \neg G) \rhd F$.
By Proposition \ref{Prop:KO}.3, $\vdash \Gamma_1 \to \left( J \rhd F \to (G \land \Box \neg G ) \rhd F \right)$. 
By Lemma \ref{l2}.1, $\vdash \Gamma_1 \to \left( J \rhd F \to G \rhd F \right)$. 
Since $\vdash \Gamma_1 \to \neg ( G \rhd F )$, we get $\vdash \Gamma_1 \to \neg ( J \rhd F)$.
From (\ref{Fml3}) and $\J{3}$, we obtain $\displaystyle \vdash \Gamma_1 \to \left( \bigvee_{j \in \kappa}I_j \land \neg F \right)\rhd F$.
By Proposition \ref{Prop:KO}.6, %$\vdash \left( J \rhd \bigvee I_j \right) \land \left( \bigvee I_j \land \neg F \rhd F \right) \to J \rhd F$. Hence
$\displaystyle \vdash \Gamma_1 \to \left( J \rhd \bigvee_{j \in \kappa} I_j \to J \rhd F \right)$. 
Hence 
\begin{equation*}
\vdash \Gamma_1 \to \neg \left(J \rhd \bigvee_{j \in \kappa} I_j \right).
\end{equation*}
From this and (\ref{Fml5}), we conclude that $\neg (J \rhd \bigvee_{j \in \kappa} I_j)$ separates $(\Gamma_1, \Gamma_2)$, a contradiction. 
Therefore $(X', Y')$ is inseparable.

Now let $\Delta \in K(\Phi^1, \Phi^2)$ be a complete pair extending $(X', Y')$. 
We have $\Gamma \prec_F \Delta$ and $G, \Box {\sim} F \in \Delta_1$. 
The other case $\neg(G \rhd F) \in \Gamma_2$ is proved in a similar way.
\end{proof}

\begin{lem}\label{L3-m2}
Let $\Gamma, \Delta \in K(\Phi^1, \Phi^2)$. 
Suppose that $\Gamma \prec_A \Delta$, $G \rhd F \in \Gamma_1 \cup \Gamma_2$ and $G \in \Delta_1 \cup \Delta_2$.
Then there exists a pair $\Theta \in K(\Phi^1, \Phi^2)$ such that:
\begin{itemize}
\item $\Gamma \prec_A \Theta$;
\item $F \in \Theta_1 \cup \Theta_2$;
%\item For $i \in \{1, 2\}$, $\Box {\sim} A, {\sim} A \in \Theta_i$ if $A \in \Phi_\rhd^i$.
\item $\Box {\sim} A, {\sim} A \in \Theta_1 \cup \Theta_2$. 
\end{itemize}
\end{lem}

\begin{proof}
Suppose $G \rhd F \in \Gamma_1$. 
From $G \in \Delta_1 \cup \Delta_2$, we obtain $G \in \Delta_1$ by the inseparability of $\Delta$. 
We distinguish the following two cases:

(Case 1): Assume $A \in \Phi^1_\rhd$. 
Then $G \rhd A \in \Phi^1$. 
If $G \rhd A \in \Gamma_1$, then ${\sim} G \in \Delta_1$ because $\Gamma \prec_A \Delta$. 
This contradicts the consistency of $\Delta_1$. 
Therefore $G \rhd A \notin \Gamma_1$. 
Since $\Gamma$ is complete, we have $\neg (G \rhd A) \in \Gamma_1$. 

Let:
\begin{align*}
	X':= & \boxdot \Gamma_1 \cup \{ \Box {\sim} F, F, \Box {\sim} A, {\sim} A \} \cup \{ \Box {\sim} B, {\sim} B : B \rhd A \in \Gamma_1 \};\\
	Y':= & \boxdot \Gamma_2 \cup \{ \Box {\sim} C, {\sim} C : C \in \Phi^2_\rhd \ \text{and for some}\  I \in \mathcal{L}_1 \cap \mathcal{L}_2, \\
	& \hspace{10em}  \vdash \Gamma_1 \to (I \land \neg A) \rhd A\ \&\ \vdash \Gamma_2 \to C \rhd I \}.
\end{align*}
We show $\Box {\sim} F \not \in \Gamma_1 \cup \Gamma_2$. 
If $\Box {\sim} G \in \Gamma_1$, then ${\sim} G \in \Delta_1$ because $\Gamma \prec \Delta$. 
This contradicts the consistency of $\Delta_1$. 
Hence $\Box {\sim} G \notin \Gamma_1$. 
Since $\vdash \Gamma_1 \to (G \rhd F) \land \Diamond G$, we have $\vdash \Gamma_1 \to \Diamond F$ by $\J{4}$.
Therefore $\Box {\sim} F \notin \Gamma_1$. 
Moreover, if $\Box {\sim} F \in \Gamma_2$, then $\Diamond F$ would separate $(\Gamma_1, \Gamma_2)$, a contradiction. 
Thus $\Box {\sim} F \notin \Gamma_2$. 

%If $\Box {\sim} F \in \Gamma_1$, then $\vdash \Gamma_1 \to \Box {\sim} F$. By Proposition \ref{Prop:KO}.2, 
%$\vdash \Box {\sim} F \to F \rhd A$. By the supposition, $\vdash \Gamma_1 \to G \rhd F$. Applying $\J{2}$, we get $\vdash \Gamma_1 \to G \rhd A$, and hence $G \rhd A \in \Gamma_1$. It follows that ${\sim} G \in \Delta_1$. This contradicts that $G \in \Delta_1$. Thus $\Box {\sim} F \not \in \Gamma_1$. Moreover, if $\Box {\sim} F \in \Gamma_2$, then we can derive that $\Diamond F$ separates $(\Gamma_1, \Gamma_2)$, a contradiction. Hence $\Box {\sim} F \not \in \Gamma_2$.
We show that $(X', Y')$ is inseparable. 
Suppose, for a contradiction, that for some $J \in \mathcal{L}_1 \cap \mathcal{L}_2$,
$\vdash X' \to J$ and $\vdash Y' \to \neg J$.

From $\vdash Y' \to \neg J$,
\begin{equation*}
\vdash \boxdot \Gamma_2 \to \left( J \to \bigvee_{j \in \kappa}{(\Diamond C_j \lor C_j)} \right), 
\end{equation*}
where $\kappa$ is an appropriate index set such that for each $j \in \kappa$, $C_j \in \Phi^2_\rhd$ and there exists a formula $I_j \in \mathcal{L}_1 \cap \mathcal{L}_2$ such that $\vdash \Gamma_1 \to (I_j \land \neg A) \rhd A$ and $\vdash \Gamma_2 \to C_j \rhd I_j$. 
Then 
\[
\vdash \Gamma_2 \to \Box \left( J \to \bigvee_{j \in \kappa}{(\Diamond C_j \lor C_j)} \right). 
\]
Since $\displaystyle \vdash \Gamma_2 \to \left(\bigvee_{j \in \kappa}{(\Diamond C_j \lor C_j)}\right) \rhd \bigvee I_j$, by Proposition \ref{Prop:KO}.2, we obtain
\begin{align}\label{Fml6}
\vdash \Gamma_2 \to J \rhd \bigvee I_j.
\end{align}

On the other hand, from $\vdash X' \to J$,
\begin{align*}
&\vdash \boxdot \Gamma_1 \to \left( \neg J \land \Box \neg F \land F \land \neg A \to \Diamond A \lor \bigvee_{B \rhd A \in \Gamma_1}{(\Diamond B \lor B)} \right), \\
&\vdash \Gamma_1 \to \Box \left( \neg J \land \Box \neg F \land F \land \neg A \to \Diamond A \lor \bigvee_{B \rhd A \in \Gamma_1}{(\Diamond B \lor B)} \right). 
\end{align*}
Then by Proposition \ref{Prop:KO}.2, we obtain 
%&\vdash \Gamma_1 \to \left( \bigvee{(\Diamond B \lor B)} \lor \Diamond A \rhd A \to \neg J \land \Box {\sim} F \land F \land {\sim} A \rhd A \right). \tag{}
$\vdash \Gamma_1 \to (\neg J \land \Box \neg F \land F \land \neg A) \rhd A$ because $\displaystyle \vdash \Gamma_1 \to \left( \Diamond A \lor \bigvee_{B \rhd A \in \Gamma_1}{(\Diamond B \lor B)} \right) \rhd A$.
By Proposition \ref{Prop:KO}.3, $\vdash \Gamma_1 \to \left( J \rhd A \to (\Box \neg F \land F \land \neg A) \rhd A \right)$.
By Lemma \ref{l2}.2, we have $\vdash \Gamma_1 \to G \rhd (\Box \neg F \land F)$. 
Then by Proposition \ref{Prop:KO}.6, we obtain $\vdash \Gamma_1 \to ((\Box \neg F \land F \land \neg A) \rhd A \to G \rhd A)$. 
Thus, $\vdash \Gamma_1 \to (J \rhd A \to G \rhd A)$. 
Since $\neg (G \rhd A) \in \Gamma_1$, we get $\vdash \Gamma_1 \to \neg (J \rhd A)$.
Since $\displaystyle \vdash \Gamma_1 \to \left(\bigvee_{j \in \kappa} I_j \land \neg A \right) \rhd A$, we have $\displaystyle \vdash \Gamma_1 \to \left( J \rhd \bigvee_{j \in \kappa} I_j \to J \rhd A \right)$ by Proposition \ref{Prop:KO}.6. 
Therefore 
\begin{align*}
\vdash \Gamma_1 \to \neg \left(J \rhd \bigvee_{j \in \kappa} I_j \right).
\end{align*}

From this and (\ref{Fml6}), we conclude that $\neg (J \rhd \bigvee_{j \in \kappa} I_j)$ separates $(\Gamma_1, \Gamma_2)$, a contradiction.

(Case 2): Assume $A \in \Phi^2_\rhd$. Let:
\begin{align*}
X':= & \boxdot \Gamma_1 \cup \{ \Box {\sim} F, F \}\\
& \cup \{ \Box {\sim} B, {\sim} B : B \in \Phi^1_\rhd\  \text{and for some}\  I \in \mathcal{L}_1 \cap \mathcal{L}_2, \\
& \hspace{7em}  \vdash \Gamma_1 \to B \rhd I \ \&\ \vdash \Gamma_2 \to (I \land \neg A) \rhd A \}; \\
Y':= & \boxdot \Gamma_2 \cup \{ \Box {\sim} A, {\sim} A \} \cup \{ \Box {\sim} C, {\sim} C : C \rhd A \in \Gamma_2 \}.
\end{align*}

As in Case 1, it can be shown $\Box {\sim} F \not \in \Gamma_1 \cup \Gamma_2$. 
We prove that $(X', Y')$ is inseparable. 
Suppose, for a contradiction, that for some $J \in \mathcal{L}_1 \cap \mathcal{L}_2$, $\vdash X' \to J$ and $\vdash Y' \to \neg J$.
From $\vdash X' \to J$,
\[
\vdash \boxdot \Gamma_1 \to \left( \Box \neg F \land F \land \neg J \to \bigvee_{j\in\kappa}{(\Diamond B_j \lor B_j)} \right), 
\]
where $\kappa$ is an appropriate index set such that for each $j \in \kappa$, $B_j \in \Phi^1_\rhd$ and there exists a formula $I_j \in \mathcal{L}_1 \cap \mathcal{L}_2$ such that $\vdash \Gamma_1 \to B_j \rhd I_j$ and $\vdash \Gamma_2 \to (I_j \land \neg A) \rhd A$. 
Then 
\[
\vdash \Gamma_1 \to \Box  \left( \Box \neg F \land F \land \neg J \to \bigvee_{j\in\kappa}{(\Diamond B_j \lor B_j)} \right). 
\]
Since $\displaystyle \vdash \Gamma_1 \to \left(\bigvee_{j \in \kappa}{(\Diamond B_j \lor B_j)} \right) \rhd \bigvee I_j$, we have 
\[
	\vdash \Gamma_1 \to (\Box \neg F \land F \land \neg J) \rhd \bigvee_{j \in \kappa} I_j
\]
by Proposition \ref{Prop:KO}.2. 
Then 
\begin{align*}
&\vdash \Gamma_1 \to \left(\Box \neg F \land F \land \bigwedge_{j \in \kappa} \neg I_j \land \neg J \right) \rhd \left( \bigvee_{j \in \kappa} I_j \lor J \right),\\
&\vdash \Gamma_1 \to \left(\Box \neg F \land F \land \neg \left( \bigvee_{j \in \kappa} I_j \lor J \right) \right) \rhd \left( \bigvee_{j \in \kappa} I_j \lor J \right).
\end{align*}
Since $G \rhd F \in \Gamma_1$, by Lemma \ref{l2}.2, we obtain $\vdash \Gamma_1 \to G \rhd (\Box \neg F \land F)$.
Therefore by Proposition \ref{Prop:KO}.6, we obtain
\begin{equation}\label{Fml7}
\vdash \Gamma_1 \to G \rhd \left(\bigvee_{j \in \kappa} I_j \lor J \right).
\end{equation}

On the other hand, from $\vdash Y' \to \neg J$,
\begin{align*}
& \vdash \boxdot \Gamma_2 \to \left( J \land \neg A \to \Diamond A \lor \bigvee_{C \rhd A \in \Gamma_2}{(\Diamond C \lor C)} \right),\\
& \vdash \Gamma_2 \to \Box \left( J \land \neg A \to \Diamond A \lor \bigvee_{C \rhd A \in \Gamma_2}{(\Diamond C \lor C)}\right). 
\end{align*}
Since $\displaystyle \vdash \Gamma_2 \to \left( \Diamond A \lor \bigvee_{C \rhd A \in \Gamma_2}{(\Diamond C \lor C)} \right) \rhd A$, we obtain $\vdash \Gamma_2 \to (J \land \neg A) \rhd A$ by Proposition \ref{Prop:KO}.2.
Since $\displaystyle \vdash \Gamma_2 \to \left(\bigvee_{j \in \kappa} I_j \land \neg A \right) \rhd A$, we have 
\begin{align*}
\vdash \Gamma_2 \to \left( \left(\bigvee_{j \in \kappa} I_j \lor J \right) \land \neg A \right) \rhd A. \end{align*}

From this and (\ref{Fml7}), we conclude ${\sim} G \in \Delta_1$ because $\Gamma \prec_A \Delta$. 
This contradicts the consistency of $\Delta_1$.

In both cases, $(X', Y')$ is inseparable, and hence we can obtain a complete pair $\Theta \in K(\Phi^1, \Phi^2)$ which extends $(X', Y')$ and satisfies the desired conditions. 
\end{proof}

\subsection{Proof of Theorem \ref{T3-1}}

We are ready to prove Theorem \ref{T3-1}. 

\begin{proof}[Proof of Theorem \ref{T3-1}]
Suppose that the implication $A_0 \to B_0$ has no interpolant, and we would like to show $\nvdash A_0 \to B_0$.
It follows that $(\{A_0\}, \{ \neg B_0\} )$ is inseparable. 
Let $\Phi^1$ (resp.~$\Phi^2$) be the smallest finite adequate set containing $A_0$ (resp.~$\neg B_0$), and put $K := K(\Phi^1, \Phi^2)$. 
There exists $\Gamma' \in K(\Phi^1, \Phi^2)$ such that $A_0 \in \Gamma_1'$ and $\neg B_0 \in \Gamma_2'$.
For $\Gamma \in K$, we define inductively the rank of $\Gamma$ (write $\rank{(\Gamma)}$) as $\rank{(\Gamma)} := \sup \{ \rank{(\Delta)} +1 : \Gamma \prec \Delta \}$, where $\sup \emptyset = 0$. This is well-defined because $\prec$ is conversely well-founded.

For finite sequences $\tau$ and $\sigma$ of formulas, let $\tau \subseteq \sigma$ denote that $\sigma$ is an end-extension of $\tau$. 
Let $\tau \ast \langle A \rangle$ be the sequence obtained from $\tau$ by adding $A$ as the last element. 

We define an $\IL^-$-model $M = \langle W, R, \{S_w\}_{w \in W}, \Vdash \rangle$ as follows:
\begin{eqnarray*}
W := \{ \langle \Gamma, \tau \rangle : \Gamma \in K \ \text{and} \ \tau \ \text{is a finite sequence of elements of}\\
\Phi_\rhd^1 \cup \Phi_\rhd^2 \ \text{with}\  \rank (\Gamma) + |\tau| \leq \rank (\Gamma') \};
\end{eqnarray*}
\begin{equation*}
\langle \Gamma, \tau \rangle R \langle \Delta, \sigma \rangle :\Leftrightarrow \Gamma \prec \Delta \ \text{and}\  \tau \subsetneq \sigma;
\end{equation*}
\begin{align*}
	& \langle \Delta, \sigma \rangle S_{\langle \Gamma, \tau \rangle} \langle \Theta, \rho \rangle \\
	&:\Leftrightarrow \left\{
		\begin{array}{l}
		\langle \Gamma, \tau \rangle R \langle \Delta, \sigma \rangle,
		\langle \Gamma, \tau \rangle R \langle \Theta, \rho \rangle \ \text{and}\ \\
		\text{if}\ \tau \ast \langle A \rangle \subseteq \sigma, \Gamma \prec_A \Delta
		\ \text{and}\ \Box {\sim} A \in \Delta_1 \cup \Delta_2, \\
		\text{then}\  \tau \ast \langle A \rangle \subseteq \rho, \Gamma \prec_A \Theta \ \text{and}\ \Box {\sim} A, {\sim} A \in \Theta_1 \cup \Theta_2; 
		\end{array}
	\right.
	\end{align*}
\begin{eqnarray*}
\langle \Gamma, \tau \rangle \Vdash p :\iff p \in \Gamma_1\cup \Gamma_2.
\end{eqnarray*}

\begin{cl}
$\IL^-(\J{2}_+, \J{5})$ is valid in the frame of $M$.
\end{cl}
\begin{proof}
It is clear that $R$ is transitive and conversely well-founded. %We verify that $S_{\langle \Gamma, \tau \rangle}$ satisfies the desired properties.
\begin{itemize}

\item Suppose $\langle \Delta, \sigma \rangle S_{\langle \Gamma, \tau \rangle} \langle \Theta, \rho \rangle$. 
Then we have $\langle \Gamma, \tau \rangle R \langle \Theta, \rho \rangle$ by the definition of $S_{\langle \Gamma, \tau \rangle}$. 
Therefore $\J{4}_+$ is valid in the frame of $M$.

\item Suppose $\langle \Delta, \sigma \rangle S_{\langle \Gamma, \tau \rangle} \langle \Theta, \rho \rangle S_{\langle \Gamma, \tau \rangle} \langle \Lambda, \pi \rangle$.
Then we have $\langle \Gamma, \tau \rangle R \langle \Delta, \sigma \rangle$ and $\langle \Gamma, \tau \rangle R \langle \Lambda, \pi \rangle$. 

Assume $\tau \ast \langle A \rangle \subseteq \sigma$, $\Gamma \prec_A \Delta$ and $\Box {\sim} A \in \Delta_1 \cup \Delta_2$. By $\langle \Delta, \sigma \rangle S_{\langle \Gamma, \tau \rangle} \langle \Theta, \rho \rangle$, we obtain
$\tau \ast \langle A \rangle \subseteq \rho$, $\Gamma \prec_A \Theta$ and $\Box {\sim} A \in \Theta_1 \cup \Theta_2$.
By $\langle \Theta, \rho \rangle S_{\langle \Gamma, \tau \rangle} \langle \Lambda, \pi \rangle$, we conclude
$\tau \ast \langle A \rangle \subseteq \pi$, $\Gamma \prec_A \Lambda$ and $\Box {\sim} A, {\sim} A \in \Lambda_1 \cup \Lambda_2$.

Thus $\langle \Delta, \sigma \rangle S_{\langle \Gamma, \tau \rangle} \langle \Lambda, \pi \rangle$.
We obtain that $\J{2}_+$ is valid in the frame of $M$. 

\item Suppose $\langle \Gamma, \tau \rangle R \langle \Delta, \sigma \rangle R \langle \Theta, \rho \rangle$.
Then $\langle \Gamma, \tau \rangle R \langle \Delta, \sigma \rangle$, and $\langle \Gamma, \tau \rangle R \langle \Theta, \rho \rangle$ by the transitivity of $R$. 

Assume $\tau \ast \langle A \rangle \subseteq \sigma$, $\Gamma \prec_A \Delta$ and $\Box {\sim} A \in \Delta_1 \cup \Delta_2$. 
Since $\sigma \subseteq \rho$, we have $\tau \ast \langle A \rangle \subseteq \rho$. 
Since $\Delta \prec \Theta$, we have $\Box {\sim} A, {\sim} A \in \Theta_1 \cup \Theta_2$. 
Also by Lemma \ref{L3-2}, $\Gamma \prec_A \Theta$.

Thus $\langle \Delta, \sigma \rangle S_{\langle \Gamma, \tau \rangle} \langle \Theta, \rho \rangle$.
We conclude that $\J{5}$ is valid in the frame of $M$. 
\end{itemize}
\end{proof}
%Truth Lemma
\begin{lem}[The Truth Lemma]\label{L3-3}
For $B \in \Phi^1 \cup \Phi^2$ and $\langle \Gamma, \tau \rangle \in W$, the following are equivalent:
\begin{enumerate}
\item $B \in \Gamma_1 \cup \Gamma_2$. 
\item $\langle \Gamma, \tau \rangle \Vdash B$.
\end{enumerate}
\end{lem}
\begin{proof}
Induction on the construction of $B$. 
We only prove for $B \equiv G \rhd F$. 

($1 \Rightarrow 2$): Assume $G \rhd F \in \Gamma_1 \cup \Gamma_2$. 
Let $\langle \Delta, \sigma \rangle \in W$ be any element such that $\langle \Gamma, \tau \rangle R \langle \Delta, \sigma \rangle$ and $\langle \Delta, \sigma \rangle \Vdash G$. 
By induction hypothesis, $G \in \Delta_1 \cup \Delta_2$. 
We distinguish the following two cases:

(Case 1): Assume that $\tau \ast \langle A \rangle \subseteq \sigma$, $\Gamma \prec_A \Delta$ and $\Box {\sim} A \in \Delta_1 \cup \Delta_2$. 
By Lemma \ref{L3-m2}, there exists a pair $\Theta \in K$ such that $\Gamma \prec_A \Theta$, $F \in \Theta_1 \cup \Theta_2$ and $\Box {\sim} A, {\sim} A \in \Theta_1 \cup \Theta_2$.

Take $\rho:= \tau \ast \langle A \rangle$. 
By $\Gamma \prec \Theta$, $\rank{(\Theta)}+1 \leq \rank{(\Gamma)}$. 
We have
\begin{equation*}
\rank{(\Theta)} + |\rho| = \rank{(\Theta)} + 1 + |\tau| \leq \rank{(\Gamma)} + |\tau| \leq \rank{(\Gamma')}.
\end{equation*}
It follows that $\langle \Theta, \rho \rangle \in W$, and we have $\langle \Delta, \sigma \rangle S_{\langle \Gamma, \tau \rangle} \langle \Theta, \rho \rangle$.
By induction hypothesis, $\langle \Theta, \rho \rangle \Vdash F$. 
Therefore $\langle \Gamma, \tau \rangle \Vdash G \rhd F$.

(Case 2): Otherwise, by Lemma \ref{L3-1}, we have $\Gamma \prec_\bot \Delta$. By Lemma \ref{L3-m2}, there exists a pair $\Theta \in K$ such that $\Gamma \prec_\bot \Theta$ and $F \in \Theta_1 \cup \Theta_2$.

Take $\rho := \tau \ast \langle \bot \rangle$. 
Then we have $\langle \Theta, \rho \rangle \in W$ by a similar argument as in Case 1. 
By the definition of $S_{\langle \Gamma, \tau \rangle}$ and induction hypothesis, $\langle \Delta, \sigma \rangle S_{\langle \Gamma, \tau \rangle} \langle \Theta, \rho \rangle$ and $\langle \Theta, \rho \rangle \Vdash F$. 
Therefore $\langle \Gamma, \tau \rangle \Vdash G \rhd F$.

($2 \Rightarrow 1$): Assume $G \rhd F \not \in \Gamma_1 \cup \Gamma_2$. 
Then $\neg (G \rhd F) \in \Gamma_1 \cup \Gamma_2$ because $\Gamma$ is complete. 
By Lemma \ref{L3-m1}, there exists a pair $\Delta \in K$ such that $\Gamma \prec_F \Delta$ and $G, \Box {\sim} F \in \Delta_1 \cup \Delta_2$. 
Let $\sigma := \tau \ast \langle F \rangle$. 
We have $\langle \Delta, \sigma \rangle \in W$. 
By induction hypothesis, $\langle \Delta, \sigma \rangle \Vdash G$. 
It suffices to show that for any $\langle \Theta, \rho \rangle \in W$, if $\langle \Delta, \sigma \rangle S_{\langle \Gamma, \tau \rangle} \langle \Theta, \rho \rangle$ then $\langle \Theta, \rho \rangle \nVdash F$.
Suppose $\langle \Delta, \sigma \rangle S_{\langle \Gamma, \tau \rangle} \langle \Theta, \rho \rangle$. 
Since $\tau \ast \langle F \rangle \subseteq \sigma$, $\Gamma \prec_F \Delta$ and $\Box {\sim} F \in \Delta_1 \cup \Delta_2$, we have ${\sim} F \in \Theta_1 \cup \Theta_2$ (and hence $F \not\in \Theta_1 \cup \Theta_2$). 
By induction hypothesis, $\langle \Theta, \rho \rangle \nVdash F$.
\end{proof}

Let $\epsilon$ be the empty sequence, then $\langle \Gamma', \epsilon \rangle \in W$ because $\rank(\Gamma') + |\epsilon| \leq \rank(\Gamma')$. 
By the Truth Lemma (Lemma \ref{L3-3}), $\langle \Gamma', \epsilon \rangle \Vdash A_0 \land \neg B_0$, and therefore $A_0 \to B_0$ is not valid in $M$. 
It follows that $\IL^-(\J{2}_+, \J{5})$ does not prove $A_0 \to B_0$. 
\end{proof}

\subsection{Consequences of Theorem \ref{T3-1}}

In this subsection, we prove some consequences of Theorem \ref{T3-1} on interpolation properties. 
First, we prove that $\IL^-(\J{2}_+, \J{5})$ has a version of the $\rhd$-interpolation property (see \cite{AHD01}). 
Secondly, we notice that CIP for $\IL$ easily follows from Theorem \ref{T3-1}. 

Before them, we show the so-called generated submodel lemma. 
Let $M = \langle W, R, \{S_w\}_{w \in W}, \Vdash \rangle$ be any $\IL^-$-model such that $\J{4}_+$ is valid in the frame of $M$. 
For each $r \in W$, we define an $\IL^-$-model $M^* = \langle W^*, R^*, \{S_w^*\}_{w \in W^*}, \Vdash^* \rangle$ as follows: 
\begin{itemize}
	\item $W^* : = \uparrow (r) \cup \{r\}$; 
	\item $x R^* y : \iff x R y$; 
	\item $y S_x^* z : \iff y S_x z$; 
	\item $x \Vdash^* p : \iff x \Vdash p$. 
\end{itemize}
We call $M^*$ the \textit{submodel of $M$ generated by $r$}. 
It is easy to show that if $\J{1}$ is valid in the frame of $M$, then it is also valid in the frame of $M^*$. 
This is also the case for $\J{2}_+$ and $\J{5}$. 
Also the following lemma is easily obtained. 

\begin{lem}[The Generated Submodel Lemma]\label{Lem:GS}
Suppose that $\J{4}_+$ is valid in the frame of an $\IL^-$-model $M = \langle W, R, \{S_w\}_{w \in W}, \Vdash \rangle$. 
For any $r \in W$, let $M^* = \langle W^*, R^*, \{S_w^*\}_{w \in W^*}, \Vdash^* \rangle$ be the submodel of $M$ generated by $r$. 
Then for any $x \in W^*$ and formula $A$, $x \Vdash A$ if and only if $x \Vdash^* A$. 
\end{lem}
\begin{proof}
This is proved by induction on the construction of $A$. 
We only prove the case $A \equiv (B \rhd C)$. 

$(\Rightarrow)$: Suppose $x \Vdash B \rhd C$. 
Let $y \in W^*$ be any element such that $x R^* y$ and $y \Vdash^* B$. 
Then $x R y$, and by induction hypothesis, $y \Vdash B$. 
Hence there exists $z \in W$ such that $y S_x z$ and $z \Vdash C$. 
Since $\J{4}_+$ is valid in the frame of $M$, $x R z$. 
Since $r R x$, we have $r R z$. 
Thus $z \in W^*$. 
It follows $y S_x^* z$. 
By induction hypothesis, $z \Vdash^* C$. 
Therefore $x \Vdash^* B \rhd C$. 

$(\Leftarrow)$: Suppose $x \Vdash^* B \rhd C$. 
Let $y \in W$ be any element with $x R y$ and $y \Vdash B$. 
Since $x \in W^*$, we have $y \in W^*$, and hence $x R^* y$. 
By induction hypothesis, $y \Vdash^* B$. 
Then for some $z \in W^*$, $y S_x^* z$ and $z \Vdash^* C$. 
We have $y S_x z$. 
By induction hypothesis, $z \Vdash C$. 
Thus we conclude $x \Vdash B \rhd C$. 
\end{proof}

\begin{prop}\label{Prop:TIP}
For any formulas $A$ and $B$, the following are equivalent: 
\begin{enumerate}
	\item $\vdash A \rhd B$. 
	\item $\vdash A \to \Diamond B$. 
\end{enumerate}
\end{prop}
\begin{proof}
$(1 \Rightarrow 2)$: Suppose $\nvdash A \to \Diamond B$. 
Then by Theorem \ref{Thm:KC}, there exist an $\IL^-$-model $M = \langle W, R, \{S_w\}_{w \in W}, \Vdash \rangle$ and $r \in W$ such that $\IL^-(\J{2}_+, \J{5})$ is valid in the frame of $M$ and $r \Vdash A \land \Box \neg B$. 
By the Generated Submodel Lemma, we may assume that $r$ is the root of $M$, that is, for all $w \in W \setminus \{r\}$, $r R w$. 

We define a new $\IL^-$-model $M' = \langle W', R', \{S_w'\}_{w \in W'}, \Vdash' \rangle$ as follows: 
\begin{itemize}
	\item $W' : = W \cup \{r_0\}$, where $r_0$ is a new element; 
	\item $x R' y : \iff \begin{cases} x R y & \text{if}\ x \neq r_0, \\ y \in W & \text{if}\ x = r_0; \end{cases}$ 
	\item $y S_x' z : \iff \begin{cases} y S_x z & \text{if}\ x \neq r_0, \\ y R z & \text{if}\ x = r_0; \end{cases}$
	\item $x \Vdash' p : \iff x \neq r_0$ and $x \Vdash p$. 
\end{itemize}
Then $\IL^-(\J{2}_+, \J{5})$ is also valid in the frame of $M'$. 
Also it is easily shown that for any $x \in W$ and any formula $C$, $x \Vdash C$ if and only if $x \Vdash' C$. 

Then $r \Vdash' A \land \Box \neg B$. 
Let $x \in W'$ be any element such that $r S_{r_0}' x$. 
Then $r R x$, and hence $r R' x$. 
We have $x \nVdash' B$. 
Therefore we obtain $r_0 \nVdash' A \rhd B$. 
It follows $\nvdash A \rhd B$. 

$(2 \Rightarrow 1)$: Suppose $\vdash A \to \Diamond B$, then $\vdash \Diamond B \rhd B \to A \rhd B$ by $\R{2}$.  
Thus $\vdash A \rhd B$. 
\end{proof}

\begin{cor}[A version of the $\rhd$-interpolation property]
Let $A$ and $B$ be any formulas. 
If $\vdash A \rhd B$, then there exists a formula $C$ such that $v(C) \subseteq v(A) \cap v(B)$, $\vdash A \to C$ and $\vdash C \rhd B$. 
\end{cor}
\begin{proof}
Suppose $\vdash A \rhd B$. 
Then by Proposition \ref{Prop:TIP}, $\vdash A \to \Diamond B$. 
By Theorem \ref{T3-1}, there exists a formula $C$ such that $v(C) \subseteq v(A) \cap v(B)$, $\vdash A \to C$ and $\vdash C \to \Diamond B$. 
By Proposition \ref{Prop:TIP} again, we obtain $\vdash C \rhd B$. 
\end{proof}

\begin{prob}
Does the logic $\IL^-(\J{2}_+, \J{5})$ have the original version of the $\rhd$-interpolation property?
That is, for every formulas $A$ and $B$ with $\vdash A \rhd B$, does there exist a formula $C$ such that $v(C) \subseteq v(A) \cap v(B)$, $\vdash A \rhd C$ and $\vdash C \rhd B$?
\end{prob}

For each formula $A$, let $\Sub(A)$ be the set of all subformulas of $A$. 
Also let $\PSub(A) : = \Sub(A) \setminus \{A\}$. 
We prove that $\IL$ is embeddable into $\IL^-(\J{2}_+, \J{5})$ in some sense. 

\begin{prop}\label{Prop:EMB}
For any formula $A$, the following are equivalent: 
\begin{enumerate}
	\item $\IL \vdash A$. 
	\item $A$ is valid in all finite $\IL^-$-frames in which all axioms of $\IL$ are valid. 
	\item $\vdash \boxdot \bigwedge\{B \rhd B : B \in \PSub(A)\} \to A$. 
\end{enumerate}
\end{prop}
\begin{proof}
$(1 \Rightarrow 2)$ is obvious. 

$(3 \Rightarrow 1)$ follows from Proposition \ref{Prop:KO}.7. 

$(2 \Rightarrow 3)$: Suppose $L \nvdash \boxdot \bigwedge\{B \rhd B : B \in \PSub(A)\} \to A$. 
Then by Theorem \ref{Thm:KC}, there exist a finite $\IL^-$-model $M = \langle W, R, \{S_w\}_{w \in W}, \Vdash \rangle$ and $r \in W$ such that $\IL^-(\J{2}_+, \J{5})$ is valid in the frame of $M$ and $r \Vdash \boxdot \bigwedge\{B \rhd B : B \in \PSub(A)\} \land \neg A$. 
By the Generated Submodel Lemma, we may assume that $r$ is the root of $M$. 

We define an $\IL^-$-model $M' = \langle W', R', \{S_w'\}_{w \in W'}, \Vdash' \rangle$ as follows: 
\begin{itemize}
	\item $W' : = W$; 
	\item $x R' y : \iff x R y$; 
	\item $y S_x' z : \iff y S_x z$ or ($x R y$ and $z = y$); 
	\item $x \Vdash' p : \iff x \Vdash p$. 
\end{itemize}

\begin{cl}
$\IL$ is valid in the frame of $M'$. 
\end{cl}
\begin{proof}
By Proposition \ref{Prop:KO}.8, it suffices to show that $\J{1}$, $\J{2}_+$ and $\J{5}$ are valid in the frame of $M'$. 

$\J{1}$: Suppose $x R y$. 
Then $y S_x' y$ by the definition of $S_x'$. 
Thus $\J{1}$ is valid. 
	
$\J{4}_+$: Suppose $y S'_x z$. 
Then $y S_x z$ or ($x R y$ and $y = z$). 
If $y S_x z$, then $x R z$ because $\J{4}_+$ is valid in the frame of $M$. 	
If $x R y$ and $y = z$, then $x R z$. 
Hence in either case, we have $x R z$. 
Therefore $\J{4}_+$ is valid. 
	
$\J{2}_+$: Suppose $y S'_x z$ and $z S'_x u$. 
We distinguish the following four cases. 
	\begin{itemize}
		\item (Case 1): $y S_x z$ and $z S_x u$. 
Since $\J{2}_+$ is valid in the frame of $M$, $y S_x u$. 
		\item (Case 2): $y S_x z$, $x R z$ and $z = u$. Then $y S_x u$. 
		\item (Case 3): $x R y$, $y = z$ and $z S_x u$. Then $y S_x u$. 
		\item (Case 4): $x R y$, $y = z$, $x R z$ and $z = u$. Then $x R y$ and $y = u$. 
	\end{itemize}
	In either case, we have $y S_x' u$. 
	Since $\J{4}_+$ is valid, we obtain that $\J{2}_+$ is valid in the frame of $M'$. 

$\J{5}$: Suppose $x R' y$ and $y R' z$. 
	Then $x R y$ and $y R z$. 
	Since $\J{5}$ is valid in the frame of $M$, $y S_x z$. 
	Then $y S_x' z$. 
	Therefore $\J{5}$ is valid. 
\end{proof}

\begin{cl}
For any $B \in \Sub(A)$ and $x \in W$, $x \Vdash B$ if and only if $x \Vdash' B$. 
\end{cl}
\begin{proof}
We prove by induction on the construction of $B$. 
We only give a proof of the case that $B$ is $C \rhd D$. 

$(\Rightarrow)$: 
Suppose $x \Vdash C \rhd D$. 
Let $y \in W$ be such that $x R y$ and $y \Vdash' C$. 
By induction hypothesis, $y \Vdash C$. 
Then there exists $z \in W$ such that $y S_x z$ and $z \Vdash D$. 
Then $y S_x' z$ and by induction hypothesis, $z \Vdash' D$. 
Therefore $x \Vdash' C \rhd D$. 	

$(\Leftarrow)$: 
Suppose $x \Vdash' C \rhd D$. 
Let $y \in W$ be such that $x R y$ and $y \Vdash C$. 
By induction hypothesis, $y \Vdash' C$. 
Hence there exists $z \in W$ such that $y S'_x z$ and $z \Vdash' D$. 
By induction hypothesis, $z \Vdash D$. 
By the definition of $S_x'$, we have either $y S_x z$ or ($x R y$ and $y = z$). 
If $y S_x z$, then $x \Vdash C \rhd D$. 
If $x R y$ and $y = z$, then $x R y$ and $y \Vdash D$. 
Here either $x = r$ or $r R w$. 
Since $D \in \PSub(A)$, we obtain $x \Vdash D \rhd D$ because $r \Vdash \boxdot \bigwedge\{B \rhd B : B \in \Sub(A)\}$. 
Thus for some $z' \in W$, $y S_x z'$ and $z' \Vdash D$. 
We conclude $x \Vdash C\ \rhd D$. 
\end{proof}

Since $r \nVdash A$, we obtain $r \nVdash' A$ by the claim. 
Thus $A$ is not valid in some finite $\IL^-$-frame in which all axioms of $\IL$ are valid. \end{proof}

\begin{proof}[Proof of Theorem \ref{Thm:AHD}]
Suppose $\IL \vdash A \to B$. 
Then by Proposition \ref{Prop:EMB}, 
\[
	\vdash \boxdot \bigwedge\{C \rhd C : C \in \PSub(A \to B)\} \to (A \to B).
\]
Since $\PSub(A \to B) = \Sub(A) \cup \Sub(B)$, we have
\[
	\vdash \boxdot \bigwedge\{C \rhd C : C \in \Sub(A)\} \land A \to \left(\boxdot \bigwedge\{C \rhd C : C \in \Sub(B)\} \to B \right).
\]
By Theorem \ref{T3-1}, there exists a formula $D$ such that $v(D) \subseteq v(A) \cap v(B)$, 
\[
	\vdash \boxdot \bigwedge\{C \rhd C : C \in \Sub(A)\} \land A \to D
\]
and
\[
	\vdash D \to \left(\boxdot \bigwedge\{C \rhd C : C \in \Sub(B)\} \to B \right).
\] 
Then by Proposition \ref{Prop:KO}.7, we obtain $\IL \vdash A \to D$ and $\IL \vdash D \to B$. 
\end{proof}

\section{The fixed point property}\label{Sec:FPP}

In this section, we investigate FPP and $\ell$FPP. 
First, we study FPP for the logic $\IL^-(\J{2}_+, \J{5})$. 
Then, we prove that $\IL^-(\J{4}, \J{5})$ has $\ell$FPP.

\subsection{FPP for $\IL^{-}(\J{2}_{+}, \J{5})$}

From Theorem \ref{T3-1} and Lemma \ref{FPP-CU}, we immediately obtain the following corollary. 

\begin{cor}[FPP for $\IL^-(\J{2}_+, \J{5})$]\label{Cor:FPP}
$\IL^-(\J{2}_+, \J{5})$ has FPP. 
\end{cor}

Moreover, we give a syntactical proof of FPP for $\IL^{-}(\J{2}_{+}, \J{5})$ by modifying de Jongh and Visser's proof of FPP for $\IL$. 
Since the Substitution Principle (Proposition \ref{p2}) holds for extensions of $\IL^-(\J{4}_+)$, as usual, it suffices to prove that every formula of the form $A(p) \rhd B(p)$ has a fixed point in $\IL^{-}(\J{2}_{+}, \J{5})$. 
As a consequence, we show that every formula $A(p)$ which is modalized in $p$ has the same fixed point in $\IL^-(\J{2}_+, \J{5})$ as given by de Jongh and Visser. 
That is, 

\begin{thm}\label{FPrhd}
For any formulas $A(p)$ and $B(p)$, $A(\top) \rhd B(\Box \neg A(\top))$ is a fixed point of $A(p) \rhd B(p)$ in $\IL^-(\J{2}_+, \J{5})$. 
\end{thm}

\begin{lem}\label{ll1}
Let $L$ be any extension of $\IL^{-}$. 
For any formulas $A$ and $B$, if $L \vdash \Box \lnot A \to (A \leftrightarrow B)$, then $L \vdash (A \land \Box \lnot A) \leftrightarrow (B \land \Box \lnot B)$. 
\end{lem}

\begin{proof}
Suppose $L \vdash \Box \lnot A \to (A \leftrightarrow B)$. 
Then, $L \vdash \Box \lnot A \to (\Box \lnot A \leftrightarrow \Box \lnot B)$ and hence $L \vdash \Box \neg A \to \Box \neg B$. 
By combining this with our supposition, we obtain
\[
L \vdash (A \land \Box \lnot A) \to (B \land \Box \lnot B). 
\]
On the other hand, $L \vdash \lnot B \to (\Box \lnot A \to \lnot A)$. 
Hence, by the axiom scheme $\G{3}$, $L \vdash \Box \lnot B \to \Box \lnot A$. 
Therefore, by our supposition, 
\[
L \vdash (B \land \Box \lnot B) \to (A \land \Box \lnot A). 
\]
\end{proof}

\begin{lem}\label{l3}
For any formulas $A$ and $C$, 
\[
\IL^-(\J{4}_{+}) \vdash (A(\top) \land \Box \lnot A(\top)) \leftrightarrow (A(A(\top) \rhd C) \land \Box \lnot A(A(\top) \rhd C)). 
\]
\end{lem}

\begin{proof}
By Proposition \ref{Prop:KO}.1, $\IL^- \vdash \Box \lnot A(\top) \to A(\top) \rhd C$. 
Therefore, we obtain $\IL^- \vdash \Box \lnot A(\top) \to (\top \leftrightarrow (A(\top) \rhd C) )$. 
Then, $\IL^- \vdash \Box \lnot A(\top) \to \boxdot(\top \leftrightarrow (A(\top) \rhd C))$. 
Therefore, by Proposition \ref{p2}.1, we obtain 
\[
\IL^-(\J{4}_{+}) \vdash \Box \lnot A(\top) \to (A(\top) \leftrightarrow A(A(\top) \rhd C)).
\]
The lemma directly follows from this and Lemma \ref{ll1}. 
\end{proof}

\begin{lem}\label{l4}
For any formulas $A$, $C$ and $D$, 
\[
\IL^-(\J{2}, \J{4}_{+}, \J{5}) \vdash (A(\top) \rhd D) \leftrightarrow (A(A(\top) \rhd C) \rhd D). 
\]
\end{lem}

\begin{proof}
By Lemma \ref{l3} and $\R{2}$, we obtain
\[
\IL^-(\J{4}_{+}) \vdash ((A(\top) \land \Box \lnot A(\top)) \rhd D) \leftrightarrow ((A(A(\top) \rhd C) \land \Box \lnot A(A(\top) \rhd C)) \rhd D). 
\]
Therefore, by Lemma \ref{l2}.1, we obtain 
\[
\IL^-(\J{2}, \J{4}_{+}, \J{5}) \vdash (A(\top)  \rhd D) \leftrightarrow (A(A(\top) \rhd C) \rhd D). 
\]
\end{proof}

\begin{lem}\label{l5}
For any formulas $B$ and $C$, $\IL^-(\J{4}_{+})$ proves
\[
(B(\Box \lnot C) \land \Box \lnot B(\Box \lnot C)) \leftrightarrow (B(C \rhd B(\Box \lnot C)) \land \Box \lnot B(C \rhd B(\Box \lnot C))).
\]
\end{lem}

\begin{proof}
Since $\IL^- \vdash \Box \lnot B(\Box \lnot C) \to \Box(\bot \leftrightarrow B(\Box \lnot C))$, 
\[
	\IL^-(\J{4}_{+}) \vdash \Box \lnot B(\Box \lnot C) \to (C \rhd \bot \leftrightarrow C \rhd B(\Box \lnot C)).
\] 
Then, by $\J{6}$, $\IL^-(\J{4}_{+}) \vdash \Box \lnot B(\Box \lnot C) \to (\Box \lnot C \leftrightarrow C \rhd B(\Box \lnot C))$ and hence $\IL^-(\J{4}_{+}) \vdash \Box \lnot B(\Box \lnot C) \to \boxdot (\Box \lnot C \leftrightarrow C \rhd B(\Box \lnot C))$. 
Therefore, by Proposition \ref{p2}.1, we obtain
\[
\IL^-(\J{4}_{+}) \vdash \Box \lnot B(\Box \lnot C) \to (B(\Box \lnot C) \leftrightarrow B(C \rhd B(\Box \lnot C))).
\]
The lemma is a consequence of this with Lemma \ref{ll1}. 
\end{proof}

\begin{lem}\label{l6}
For any formulas $B$, $C$ and $D$, 
\[
\IL^-(\J{2}_{+}, \J{5}) \vdash (D \rhd B(\Box \lnot C)) \leftrightarrow (D \rhd B(C \rhd B(\Box \lnot C))).
\]
\end{lem}

\begin{proof}
By Lemma \ref{l5} and $\R{1}$, $\IL^-(\J{4}_{+})$ proves
\[
(D \rhd (B(\Box \lnot C) \land \Box \lnot B(\Box \lnot C))) \leftrightarrow (D \rhd (B(C \rhd B(\Box \lnot C)) \land \Box \lnot B(C \rhd B(\Box \lnot C)))).
\]
Therefore, by Lemma \ref{l2}.2, 
\[
\IL^-(\J{2}_{+}, \J{5}) \vdash (D \rhd B(\Box \lnot C)) \leftrightarrow (D \rhd B(C \rhd B(\Box \lnot C))).
\]
\end{proof}

\begin{proof}[Proof of Theorem \ref{FPrhd}]
Let $F \equiv A(\top) \rhd B(\Box \lnot A(\top))$. 
By Lemma \ref{l4} for $C \equiv D \equiv B(\Box \lnot A(\top))$, we obtain 
\[ 
\IL^-(\J{2}, \J{4}_+, \J{5}) \vdash F \leftrightarrow (A(F) \rhd B(\Box \lnot A(\top))). 
\]
Furthermore, by Lemma \ref{l6} for $C \equiv A(\top)$ and $D \equiv F$, 
\[
\IL^-(\J{2}_+, \J{5}) \vdash (A(F) \rhd B(\Box \lnot A(\top))) \leftrightarrow (A(F) \rhd B(F)).
\]
We conclude  
\[
\IL^-(\J{2}_+, \J{5}) \vdash F \leftrightarrow A(F) \rhd B(F).
\]
\end{proof}

\subsection{$\ell$FPP for $\IL^{-}(\J{4}, \J{5})$}

From Lemma \ref{l4}, we immediately obtain the following corollary. 

\begin{cor}
For any formulas $A(p)$ and $B$, if $p \notin v(B)$, then $A(\top) \rhd B$ is a fixed point of $A(p) \rhd B$ in $\IL^-(\J{2}, \J{4}_+, \J{5})$. 
\end{cor}

Therefore $\IL^-(\J{2}, \J{4}_+, \J{5})$ has $\ell$FPP. 
Moreover, we prove the following theorem. 

\begin{thm}[$\ell$FPP for $\IL^-(\J{4}, \J{5})$]\label{Thm:lFPP}
For any formulas $A(p)$ and $B$, if the formula $A(p) \rhd B$ is left-modalized in $p$, then $A(\Box \neg A(\top)) \rhd B$ is a fixed point of $A(p) \rhd B$ in $\IL^-(\J{4}, \J{5})$. 
Therefore $\IL^-(\J{4}, \J{5})$ has $\ell$FPP. 
\end{thm}

Before proving Theorem \ref{Thm:lFPP}, we prepare two lemmas. 

\begin{lem}\label{FPbox}
For any formula $A(p)$ such that $\Box A(p)$ is left-modalized in $p$, 
\[
\IL^- \vdash \Box A(\top) \leftrightarrow \Box A(\Box A(\top)). 
\]
\end{lem}
\begin{proof}
This is proved in a usual way by using Proposition \ref{p'2}. 
\end{proof}

\begin{lem}\label{Lem:lFPP}
Let $A(p)$ and $B$ be any formulas such that for any subformula $D \rhd E$ of $A(p)$, $p \notin v(E)$. 
Then 
\[
	\IL^-(\J{4}, \J{5}) \vdash (A(\Box \neg A(p)) \rhd B) \leftrightarrow (A(A(p) \rhd B) \rhd B). 
\]
\end{lem}
\begin{proof}
By Proposition \ref{Prop:KO}.1, $\IL^- \vdash \Box \neg A(p) \to A(p) \rhd B$. 
On the other hand, since $\IL^-(\J{4}) \vdash A(p) \rhd B \to (\Diamond A(p) \to \Diamond B)$, we have $\IL^-(\J{4}) \vdash \Box \neg B \to (A(p) \rhd B \to \Box \neg A(p))$. 
Hence $\IL^-(\J{4}) \vdash \Box \neg B \to (\Box \neg A(p) \leftrightarrow A(p) \rhd B)$. 
Then 
\[
	\IL^-(\J{4}) \vdash \Box \neg B \to \boxdot (\Box \neg A(p) \leftrightarrow A(p) \rhd B).
\] 
By Proposition \ref{p'2}.1, we obtain
\[
	\IL^-(\J{4}) \vdash \Box \neg B \to (A(\Box \neg A(p)) \leftrightarrow A(A(p) \rhd B)).
\] 
Thus 
\[
	\IL^-(\J{4}) \vdash (A(\Box \neg A(p)) \lor \Diamond B) \leftrightarrow (A(A(p) \rhd B) \lor \Diamond B).
\]
By $\R{2}$, we obtain 
\[
	\IL^-(\J{4}) \vdash ((A(\Box \neg A(p)) \lor \Diamond B) \rhd B) \leftrightarrow ((A(A(p) \rhd B) \lor \Diamond B) \rhd B).
\]
Therefore, we conclude 
\[
	\IL^-(\J{4}, \J{5}) \vdash (A(\Box \neg A(p)) \rhd B) \leftrightarrow (A(A(p) \rhd B) \rhd B).
\]
\end{proof}

\begin{proof}[Proof of Theorem \ref{Thm:lFPP}]
Let $F : \equiv \Box \neg A(\top)$. 
Since $\Box \neg A(p)$ is left-modalized in $p$, $\IL^- \vdash F \leftrightarrow \Box \neg A(F)$ by Lemma \ref{FPbox}. 
Since $\IL^-(\J{4}) \vdash \Box (F \leftrightarrow \Box \neg A(F))$, by Proposition \ref{p'2}.2, we have 
\[
	\IL^-(\J{4}) \vdash (A(F) \rhd B) \leftrightarrow (A(\Box \neg A(F)) \rhd B).
\] 
By Lemma \ref{Lem:lFPP}, 
\[
	\IL^-(\J{4}, \J{5}) \vdash (A(\Box \neg A(F)) \rhd B) \leftrightarrow (A(A(F) \rhd B) \rhd B). 
\]
Therefore, 
\[
	\IL^-(\J{4}, \J{5}) \vdash (A(F) \rhd B) \leftrightarrow (A(A(F) \rhd B) \rhd B). 
\]
\end{proof}

\section{Failure of $\ell$FPP, FPP and CIP}\label{Sec:CE}

In this section, we provide counter models of $\ell$FPP for $\CL$ and $\IL^-(\J{1}, \J{5})$, and also provide a counter model of FPP for $\IL^-(\J{1}, \J{4}_+, \J{5})$. 
%First, we give a counter model of $\ell$FPP for $\CL$. 
We also show that CIP is not the case for our sublogics except for $\IL^-(\J{2}_+, \J{5})$ and $\IL$. 
Let $\omega$ be the set $\{0, 1, 2, \ldots\}$ of all natural numbers.

\subsection{A counter model of $\ell$FPP for $\CL$}

In this subsection, we prove that $\IL^-$, $\IL^-(\J{1})$, $\IL^-(\J{4}_+)$, $\IL^-(\J{1}, \J{4}_+)$, $\IL^-(\J{2}_+)$ and $\CL$ have neither $\ell$FPP nor CIP.

\begin{thm}\label{c3}
The formula $p \rhd q$ which is left-modalized in $p$ has no fixed points in $\CL$. 
That is, for any formula $A$ which satisfies $v(A) \subseteq \{q\}$, 
\[
\CL \nvdash A \leftrightarrow A \rhd q. 
\]
\end{thm}

\begin{proof}
We define an $\IL^-$-frame $\mathcal{F} = \seq{W, R, \{S_{w}\}_{w \in W}}$ as follows: 
\begin{itemize}
	\item $W := \{x_{i}, y_{i} : i \in \omega\}$;
	\item $R := \{\seq{x_{i}, x_{j}}, \seq{x_{i}, y_{j}}, \seq{y_{i}, x_{j}}, \seq{y_{i}, y_{j}} \in W^2 : i > j \}$;
	\item For each $w_i \in W$ where $w \in \{x, y\}$, $S_{w_i}:=\{\seq{a, a} : w_iRa\} \cup \{\seq{a, b}:$ there exists an even number $k < i-1$ such that $((a = x_{k}$ or $a=y_{k})$ and $b = x_{k+1})\}$. 
\end{itemize}

For example, $S_{x_3}$, $S_{y_3}$, $S_{x_4}$ and $S_{y_4}$ are shown in the following figure (Figure \ref{Fig3}). 

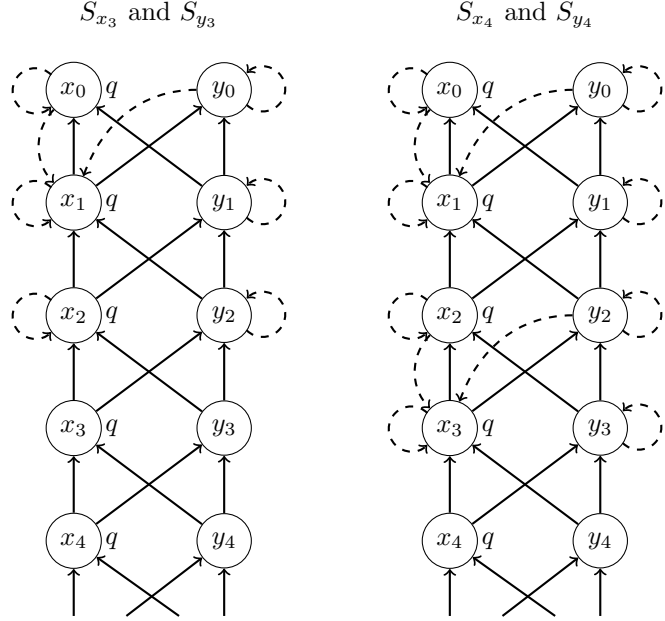
\begin{figure}[th]
\centering
\begin{tikzpicture}

\draw (1, 7) node {$S_{x_3}$ and $S_{y_3}$};

\node [draw, circle] (x_4) at (0,0) {$x_4$};
\node [draw, circle] (x_3) at (0,1.5) {$x_3$};
\node [draw, circle] (x_2) at (0,3) {$x_2$};
\node [draw, circle] (x_1) at (0,4.5) {$x_1$};
\node [draw, circle] (x_0) at (0,6) {$x_0$};
\node [draw, circle] (y_4) at (2,0) {$y_4$};
\node [draw, circle] (y_3) at (2,1.5) {$y_3$};
\node [draw, circle] (y_2) at (2,3) {$y_2$};
\node [draw, circle] (y_1) at (2,4.5) {$y_1$};
\node [draw, circle] (y_0) at (2,6) {$y_0$};

\draw [thick, ->] (x_4)--(x_3);
\draw [thick, ->] (x_3)--(x_2);
\draw [thick, ->] (x_2)--(x_1);
\draw [thick, ->] (x_1)--(x_0);
\draw [thick, ->] (y_4)--(y_3);
\draw [thick, ->] (y_3)--(y_2);
\draw [thick, ->] (y_2)--(y_1);
\draw [thick, ->] (y_1)--(y_0);
\draw [thick, ->] (x_4)--(y_3);
\draw [thick, ->] (x_3)--(y_2);
\draw [thick, ->] (x_2)--(y_1);
\draw [thick, ->] (x_1)--(y_0);
\draw [thick, ->] (y_4)--(x_3);
\draw [thick, ->] (y_3)--(x_2);
\draw [thick, ->] (y_2)--(x_1);
\draw [thick, ->] (y_1)--(x_0);
\draw [thick, ->] (0, -1)--(x_4);
\draw [thick, ->] (2, -1)--(y_4);
\draw [thick, ->] (0.7, -1)--(y_4);
\draw [thick, ->] (1.4, -1)--(x_4);

\draw [thick, ->, dashed] (-0.3, 6.2) arc (45:315:0.3);
\draw [thick, ->, dashed] (2.3, 5.8) arc (-135:135:0.3);
\draw [thick, ->, dashed] (-0.3, 4.7) arc (45:315:0.3);
\draw [thick, ->, dashed] (2.3, 4.3) arc (-135:135:0.3);
\draw [thick, ->, dashed] (-0.3, 3.2) arc (45:315:0.3);
\draw [thick, ->, dashed] (2.3, 2.8) arc (-135:135:0.3);
\draw [thick, ->, dashed] (x_0) to [out=-135, in=135] (x_1);
\draw [thick, ->, dashed] (y_0) to [out=180, in=70] (x_1);

\draw (0.3, 0) node[right] {$q$};
\draw (0.3, 1.5) node[right] {$q$};
\draw (0.3, 3) node[right] {$q$};
\draw (0.3, 4.5) node[right] {$q$};
\draw (0.3, 6) node[right] {$q$};

\draw (6, 7) node {$S_{x_4}$ and $S_{y_4}$};

\node [draw, circle] (a_4) at (5,0) {$x_4$};
\node [draw, circle] (a_3) at (5,1.5) {$x_3$};
\node [draw, circle] (a_2) at (5,3) {$x_2$};
\node [draw, circle] (a_1) at (5,4.5) {$x_1$};
\node [draw, circle] (a_0) at (5,6) {$x_0$};
\node [draw, circle] (b_4) at (7,0) {$y_4$};
\node [draw, circle] (b_3) at (7,1.5) {$y_3$};
\node [draw, circle] (b_2) at (7,3) {$y_2$};
\node [draw, circle] (b_1) at (7,4.5) {$y_1$};
\node [draw, circle] (b_0) at (7,6) {$y_0$};

\draw [thick, ->] (a_4)--(a_3);
\draw [thick, ->] (a_3)--(a_2);
\draw [thick, ->] (a_2)--(a_1);
\draw [thick, ->] (a_1)--(a_0);
\draw [thick, ->] (b_4)--(b_3);
\draw [thick, ->] (b_3)--(b_2);
\draw [thick, ->] (b_2)--(b_1);
\draw [thick, ->] (b_1)--(b_0);
\draw [thick, ->] (a_4)--(b_3);
\draw [thick, ->] (a_3)--(b_2);
\draw [thick, ->] (a_2)--(b_1);
\draw [thick, ->] (a_1)--(b_0);
\draw [thick, ->] (b_4)--(a_3);
\draw [thick, ->] (b_3)--(a_2);
\draw [thick, ->] (b_2)--(a_1);
\draw [thick, ->] (b_1)--(a_0);
\draw [thick, ->] (5, -1)--(a_4);
\draw [thick, ->] (7, -1)--(b_4);
\draw [thick, ->] (5.7, -1)--(b_4);
\draw [thick, ->] (6.4, -1)--(a_4);

\draw [thick, ->, dashed] (4.7, 6.2) arc (45:315:0.3);
\draw [thick, ->, dashed] (7.3, 5.8) arc (-135:135:0.3);
\draw [thick, ->, dashed] (4.7, 4.7) arc (45:315:0.3);
\draw [thick, ->, dashed] (7.3, 4.3) arc (-135:135:0.3);
\draw [thick, ->, dashed] (4.7, 3.2) arc (45:315:0.3);
\draw [thick, ->, dashed] (7.3, 2.8) arc (-135:135:0.3);
\draw [thick, ->, dashed] (4.7, 1.7) arc (45:315:0.3);
\draw [thick, ->, dashed] (7.3, 1.3) arc (-135:135:0.3);

\draw [thick, ->, dashed] (a_0) to [out=-135, in=135] (a_1);
\draw [thick, ->, dashed] (b_0) to [out=180, in=70] (a_1);
\draw [thick, ->, dashed] (a_2) to [out=-135, in=135] (a_3);
\draw [thick, ->, dashed] (b_2) to [out=180, in=70] (a_3);

\draw (5.3, 0) node[right] {$q$};
\draw (5.3, 1.5) node[right] {$q$};
\draw (5.3, 3) node[right] {$q$};
\draw (5.3, 4.5) node[right] {$q$};
\draw (5.3, 6) node[right] {$q$};

\end{tikzpicture}
\caption{A counter model of $\ell$FPP for $\CL$}\label{Fig3}
\end{figure}

%Then $R$ is transitive and conversely well-founded. 
%For each $w, a, b \in W$, if $a S_{w} b$, then $w R a$. 
%Thus, $\mathcal{F}$ is an $\IL^{-}$-frame. 

It is easy to show that $\J{1}$ and $\J{2}_+$ are valid in $\mathcal{F}$. 
Thus $\CL$ is valid in $\mathcal{F}$ by Proposition \ref{Prop:KO}.9. 
Let $\Vdash$ be a satisfaction relation on $\mathcal{F}$ such that for any $i \in \omega$, $x_i \Vdash q$ and $y_i \nVdash q$. 
For each $w \in W$, we say that $i \in \omega$ is an \textit{index} of $w$ if either $w = x_i$ or $w = y_i$. 

%\begin{cl}\label{CLFC}
%$\CL$ is valid in $\mathcal{F}$.
%\end{cl}
%\begin{proof}
%We confirm the following properties.

%\begin{enumerate}
%	\item For any $w, a \in W$, $wRx \Rightarrow a S_{w} a$.
%	\item For any $w, a, b \in W$, $a S_{w} b \Rightarrow wRb$.
%	\item For any $w, a, b, c \in W$, $a S_{w} b$ and $b S_{w} c \Rightarrow a S_{w} c$.
%\end{itemize}

%\noindent
%1. Suppose $w, a \in W$ and $w R a$.  Obviously $a S_{w} a$. 

%\vspace{2mm}
%\noindent
%2.  Suppose $w, a, b \in W$ and $a S_{w} b$. 
%Then, the index of $b$ is less than the index of $w$ by the definition of $S_{w}$. Therefore, $w R b$. 

%\vspace{2mm}
%\noindent
%3. Suppose $w, a, b, c \in W$, $i$ is the index of $w$, $a S_{w} b$ and $b S_{w} c$. 
%Then $a = b$ or the index of $a$ is an even number less than $i-1$. 
%If $a = b$, then $a S_{w} c$ is obvious. 
%Assume the index of $a$ is an even number less than $i-1$. 
%Then, the index of $b$ is an odd number by the definition of $S_{w}$. 
%Hence, $b = c$ by  the definition of $S_{w}$ and $b S_{w} c$. 
%Therefore, we obtain $a S_{w} c$. 

%By the Proposition \ref{Prop:FC}.1 and \ref{Prop:FC}.2, the axioms of $\CL$ are valid in $\mathcal{F}$. 
%\end{proof}

%For each $w \in W$, we define the $\IL^{-}$-model $\seq{W, R, \{S_{w}\}_{w \in W}, \Vdash}$ with $w \Vdash r :\Leftrightarrow$ there exists $i \in \omega$ such that $w = x_{i}$. 

\begin{cl}\label{CLcl1}
For any formula $A$ with $v(A) \subseteq \{q\}$, there exists an $n \in \omega$ satisfying the following two conditions:
\begin{enumerate}
	\item Either $\forall m \geq n\, (x_{m} \Vdash A)$ or $\forall m \geq n\, (x_{m} \nVdash A)$;
	\item Either $\forall m \geq n\, (y_{m} \Vdash A)$ or $\forall m \geq n\, (y_{m} \nVdash A)$. 
\end{enumerate}
\end{cl}

\begin{proof}
We prove by induction on the construction of $A$. 

\vspace{2mm}
\noindent
$A \equiv \bot$: Then $\forall m \geq 0\, (x_m \nVdash A$ and $y_m \nVdash A)$.

\vspace{2mm}
\noindent
$A \equiv q$: Then $\forall m \geq 0\, (x_m \Vdash q$ and $y_m \nVdash q)$.

\vspace{2mm}
\noindent
$A \equiv B \to C$: 
By induction hypothesis, there exist $n_{1}, n_{2} \in W$ satisfying the statement of the claim for $B$ and $C$, respectively. 
Let $n = \max\{n_1, n_2\}$. 
We distinguish the following three cases.
\begin{itemize}
\item $\forall m \geq n\, (x_m \nVdash B)$: Then $\forall m \geq n\, (x_m \Vdash B \to C)$. 
\item $\forall m \geq n\, (x_m \Vdash C)$: Then $\forall m \geq n\, (x_m \Vdash B \to C)$. 
\item $\forall m \geq n\, (x_m \Vdash B)$ and $\forall m \geq n\, (x_m \nVdash C)$: Then $\forall m \geq n\, (x_m \nVdash B \to C)$. 
\end{itemize}
In a similar way, it is proved that either $\forall m \geq n\, (y_{m} \Vdash B \to C)$ or $\forall m \geq n\, (y_{m} \nVdash B \to C)$. 

\vspace{2mm}
\noindent
$A \equiv \Box B$: We distinguish the following two cases.
\begin{itemize}
\item There exists an $n \in W$ such that either $x_n \nVdash B$ or $y_n \nVdash B$: Then $\forall m \geq n + 1\, (x_m \nVdash \Box B$ and $y_m \nVdash \Box B)$. 

\item For all $n \in W$, $x_n \Vdash B$ and $y_n \Vdash B$: Then $\forall m \geq 0\, (x_m \Vdash \Box B$ and $y_m \Vdash \Box B)$. 
\end{itemize}

\vspace{2mm}
\noindent
$A \equiv B \rhd C$: We distinguish the following five cases. 

\begin{itemize}
	\item (Case 1): 
There exists an even number $k$ such that $x_{k} \Vdash B$, $x_{k} \nVdash C$ and $x_{k+1} \nVdash C$.  
Let $m \geq k+2$. 
Then, $x_{m} R x_{k}$ and $x_{k} \Vdash B$. 
For any $v \in W$ which satisfies $x_{k}S_{x_{m}} v$, either $v = x_{k}$ or $v = x_{k+1}$ by the definition of $S_{x_{m}}$. 
Thus, $v \nVdash C$. 
Therefore, we obtain $x_{m} \nVdash B \rhd C$. 
Since $y_{m} R x_{k+1}$, we also obtain $y_{m} \nVdash B \rhd C$ in a similar way. 

	\item (Case 2): 
There exists an even number $k$ such that $y_{k} \Vdash B$, $y_{k} \nVdash C$ and $x_{k+1} \nVdash C$. 
It is proved that $k+2$ witnesses the claim as in Case 1. 

	\item (Case 3): 
There exists an odd number $k$ such that $x_{k} \Vdash B$ and $x_{k} \nVdash C$. 
Let $m \geq k+1$. 
Then, $x_{m} R x_{k}$ and $x_{k} \Vdash B$. 
For any $v \in W$ satisfying $x_{k}S_{x_{m}} v$, $v = x_{k}$ by the definition of $S_{x_{m}}$. 
Thus, $v \nVdash C$. 
Therefore, we obtain $x_{m} \nVdash B \rhd C$. 
Since $y_{m} R x_{k}$, $y_{m} \nVdash B \rhd C$ is also proved. 

	\item (Case 4): There exists an odd number $k$ such that $y_{k} \Vdash B$ and $y_{k} \nVdash C$. 
It is proved that $k+1$ witnesses the claim as in Case 3. 

	\item (Case 5): Otherwise, all of the following conditions are satisfied. 
\begin{itemize}
	\item[(I)] For any even number $k$, if $x_{k} \Vdash B$, then either $x_{k} \Vdash C$ or $x_{k+1} \Vdash C$. 
	\item[(II)] For any even number $k$, if $y_{k} \Vdash B$, then either $y_{k} \Vdash C$ or $x_{k+1} \Vdash C$. 
	\item[(III)] For any odd number $k$, if $x_{k} \Vdash B$, then $x_{k} \Vdash C$. 
	\item[(IV)] For any odd number $k$, if $y_{k} \Vdash B$, then $y_{k} \Vdash C$. 
\end{itemize}

By induction hypothesis, there exists an $n_{0} \in \omega$ which is a witness of the statement of the claim for $B$. 
We define a natural number $n$ so that for any $z \in W$ with the index $i$, if $i \geq n-1$, then $z \Vdash \neg B \lor C$. 
We distinguish the following four cases. 
 
\begin{itemize}
	\item $\forall m \geq n_{0}\, (x_{m} \Vdash B$ and $y_{m} \Vdash B)$: 
Then, by (III) and (IV), there are infinitely many odd numbers $k$ such that $x_{k} \Vdash C$ and $y_{k} \Vdash C$. 
Thus, by induction hypothesis, there exists an $n_{1} \in \omega$ such that $\forall m \geq n_{1}\, (x_{m} \Vdash C$ and $y_{m} \Vdash C)$. 
Then, we define $n := \max\{n_{0}, n_{1}\} + 1$. 

	\item $\forall m \geq n_{0}\, (x_{m} \Vdash B$ and $y_{m} \nVdash B)$:
Then, by (III), there are infinitely many odd numbers $k$ such that $x_{k} \Vdash C$. 
Thus, by induction hypothesis, there exists an $n_{1} \in \omega$ such that $\forall m \geq n_{1}\, (x_{m} \Vdash C)$. 
Then, we define $n := \max\{n_{0}, n_{1}\} + 1$. 

	\item $\forall m \geq n_{0}\, (x_{m} \nVdash B$ and $y_{m} \Vdash B)$: 
Then, by (IV), there are infinitely many odd numbers $k$ such that $y_{k} \Vdash C$. 
Thus, by induction hypothesis, there exists an $n_{1} \in \omega$ such that $\forall m \geq n_{1}\, (y_{m} \Vdash C)$. 
Then, we define $n : = \max\{n_{0}, n_{1}\} + 1$. 

	\item $\forall m \geq n_{0}\, (x_{m} \nVdash B$ and $y_{m} \nVdash B)$: 
We define $n : = n_{0} + 1$. 
\end{itemize}

Let $m \geq n$ and $z \in W$ be such that $x_{m} R z$ and $z \Vdash B$. 
We show that there exists a $v \in W$ such that $z S_{x_{m}} v$ and $v \Vdash C$. 
Let $i$ be an index of $z$. 
If $i$ is odd, then $z S_{x_{m}} z$ and $z \Vdash C$ by (III) and (IV). 
Assume that $i$ is even. 
We distinguish the following two cases. 
\begin{itemize}
	\item $n-1 \leq i < m$: 
We obtain $z \Vdash \lnot B \lor C$ by the definition of $n$. 
Since $z \Vdash B$, $z \Vdash C$. 
By the definition of $S_{x_{m}}$, $z S_{x_{m}} z$. 

	\item $i < n-1$: Then $i < m-1$. 
Therefore $z S_{x_{m}} z$ and $z S_{x_{m}} x_{i+1}$. 
Furthermore, by (I) and (II), we obtain $z \Vdash C$ or $x_{i+1} \Vdash C$. 
\end{itemize}
In any case, there exists $v \in W$ such that $z S_{x_{m}} v$ and $v \Vdash C$. 
Therefore, we obtain $x_{m} \Vdash B \rhd C$. 
Similarly, we have $y_m \Vdash B \rhd C$.  
\end{itemize}
\end{proof}

We suppose, towards a contradiction, that there exists a formula $A$ such that $v(A) \subseteq \{q\}$ and $\CL \vdash A \leftrightarrow A \rhd q$. 
Since $\CL$ is valid in $\mathcal{F}$, $A \leftrightarrow A \rhd q$ is valid in $\mathcal{F}$. 
Moreover, the following claim holds. 

\begin{cl}
For any $w \in W$ whose index is $n$, $n$ is even if and only if $w \Vdash A$. 
\end{cl}

\begin{proof}
We prove by induction on $n$. Let $w \in W$ be any element whose index is $n$. 

For $n =0$, since there is no $w' \in W$ such that $w R w'$, we obtain $w \Vdash A \rhd  q$ and hence, $w \Vdash A$. 
Suppose $n > 0$ and that the claim holds for any natural number less than $n$. 

$(\Leftarrow)$: Assume that $n$ is an odd number. 
Then $w R y_{n-1}$. 
Since $n-1$ is even, $y_{n-1} \Vdash A$ by induction hypothesis. 
Let $v$ be any element in $W$ satisfying $y_{n-1} S_{w} v$. 
By the definitions of $S_{w}$ and $\Vdash$, we obtain $v = y_{n-1}$ and $v \nVdash q$. 
Therefore, $w \nVdash A \rhd q$ and hence $w \nVdash A$. 

$(\Rightarrow)$: Assume that $n$ is an even number. 
Let $v$ be any element in $W$ with $w R v$ and $v \Vdash A$. 
Let $m$ be the index of $v$. 
Since $m < n$ and $v \Vdash A$, $m$ is even by induction hypothesis. 
Since $n$ is also even, $m < n-1$ and hence $v S_{w} x_{m+1}$. 
Furthermore, $x_{m+1} \Vdash q$ by the definition of $\Vdash$. 
Therefore, we obtain $w \Vdash A \rhd q$ and hence, $w \Vdash A$. 
\end{proof}

This contradicts Claim \ref{CLcl1}. 
Therefore, for any formula $A$ with $v(A) \subseteq \{q\}$, we obtain $\CL \nvdash A \leftrightarrow A \rhd q$. 
\end{proof}

\begin{cor}\label{Cor:15}
Let $L$ be any logic such that $\IL^{-} \subseteq L \subseteq \CL$. 
Then $L$ has neither $\ell$FPP nor CIP. 
\end{cor}

\begin{proof}
By Theorem \ref{c3}, every sublogic of $\CL$ does not have $\ell$FPP. 
By Lemma \ref{lFPP-ClU}, every logic $L$ such that $\IL^- \subseteq L \subseteq \CL$ does not have CIP. 
\end{proof}

\subsection{A counter model of $\ell$FPP for $\IL^{-}(\J{1}, \J{5})$}

In this subsection, we prove that $\IL^-(\J{5})$ and $\IL^-(\J{1}, \J{5})$ have neither $\ell$FPP nor CIP.

\begin{thm}\label{c2}
The formula $p \rhd q$ which is left-modalized in $p$ has no fixed point in $\IL^{-}(\J{1}, \J{5})$. 
That is, for any formula $A$ which satisfies $v(A) \subseteq \{q\}$, 
\[
\IL^{-}(\J{1}, \J{5}) \nvdash A \leftrightarrow A \rhd q.
\]
\end{thm}

\begin{proof}
We define an $\IL^-$-frame $\mathcal{F} = \seq{W, R, \{S_{w}\}_{w \in W}}$ as follows: 
\begin{itemize}
	\item $W := \omega \cup \{v\}$;
	\item $R := \{\seq{x, y} \in W^2 : x, y \in \omega$ and $x > y \}$;
	\item $S_{v}  := \emptyset$ and for each $n \in \omega$, $S_{n}:=\{\seq{x, y} \in W^2 : n R x$ and ($y = x$ or $x R y$ or ($x$ is even, $x < n -1$ and $y = v)) \}$. 
\end{itemize}

%By the definition of $\mathcal{F}$, $R$ is transitive and conversely well-founded and for each $w, x, y \in W$, if $x S_{w} y$, then $w R x$ because $w > x$. Thus,  $\mathcal{F}$ is an $\IL^{-}$-frame. 

For instance, the relations $S_3$ and $S_4$ are shown in the following figure (Figure \ref{Fig4}). 
In the case of $x R y$ for $x, y < n$, $x S_n y$ holds, and the corresponding broken lines are omitted in the figure. 

\begin{figure}[th]
\centering
\begin{tikzpicture}

\draw (1, 6) node {$S_{3}$};

\node [draw, circle] (4) at (0,0) {$4$};
\node [draw, circle] (3) at (0,1.2) {$3$};
\node [draw, circle] (2) at (0,2.4) {$2$};
\node [draw, circle] (1) at (0,3.6) {$1$};
\node [draw, circle] (0) at (0,4.8) {$0$};
\node [draw, circle] (v) at (2,4.8) {$v$};

\draw [thick, ->] (4)--(3);
\draw [thick, ->] (3)--(2);
\draw [thick, ->] (2)--(1);
\draw [thick, ->] (1)--(0);
\draw [thick, ->] (0, -1)--(4);

\draw [thick, ->, dashed] (-0.2, 5) arc (45:315:0.3);
\draw [thick, ->, dashed] (-0.2, 3.8) arc (45:315:0.3);
\draw [thick, ->, dashed] (-0.2, 2.6) arc (45:315:0.3);
\draw [thick, ->, dashed] (0)--(v);

\draw (2.3, 4.8) node[right] {$q$};

\draw (6, 6) node {$S_{4}$};

\node [draw, circle] (4') at (5,0) {$4$};
\node [draw, circle] (3') at (5,1.2) {$3$};
\node [draw, circle] (2') at (5,2.4) {$2$};
\node [draw, circle] (1') at (5,3.6) {$1$};
\node [draw, circle] (0') at (5,4.8) {$0$};
\node [draw, circle] (v') at (7,4.8) {$v$};

\draw [thick, ->] (4')--(3');
\draw [thick, ->] (3')--(2');
\draw [thick, ->] (2')--(1');
\draw [thick, ->] (1')--(0');
\draw [thick, ->] (5, -1)--(4');

\draw [thick, ->, dashed] (4.8, 5) arc (45:315:0.3);
\draw [thick, ->, dashed] (4.8, 3.8) arc (45:315:0.3);
\draw [thick, ->, dashed] (4.8, 2.6) arc (45:315:0.3);
\draw [thick, ->, dashed] (4.8, 1.4) arc (45:315:0.3);
\draw [thick, ->, dashed] (0')--(v');
\draw [thick, ->, dashed] (2')--(v');

\draw (7.3, 4.8) node[right] {$q$};

\end{tikzpicture}
\caption{A counter model of $\ell$FPP for $\IL^-(\J{1}, \J{5})$}\label{Fig4}
\end{figure}

%\begin{cl}\label{15FC}
%The axioms of $\IL^{-}(\J{1}, \J{5})$ are valid in $\mathcal{F}$.
%\end{cl}

%\begin{proof}

%We confirm the following properties.

%\begin{enumerate}
%	\item For any $w, x \in W$, $wRx \Rightarrow x S_{w} x$.
%	\item For any $w, x, y \in W$, $w R x$ and $x R y \Rightarrow xS_{w} y$.
%\end{enumerate}

%\noindent
%1. Suppose $w, x \in W$ and $w R x$.  Then $w, x \in \omega$ and $x < w$ by the definition of $R$. Since $x \leq x$, we obtain $x S_{w} x$ 

%\vspace{2mm}
%\noindent
%2. Suppose $w, x, y \in W$, $w R x$ and $x R y$. Then $w, x, y \in \omega$ and $y < x < w$ by the definition of $R$. Therefore, $x S_{w} y$. 

%\vspace{2mm}

%By the Proposition \ref{Prop:FC}.1, \ref{Prop:FC}.4, the axioms of $\IL^{-}(\J{1}, \J{5})$ are valid in $\mathcal{F}$. 
%\end{proof}

Then $\IL^-(\J{1}, \J{5})$ is valid in $\mathcal{F}$. 
%For each $w \in W$, we define the $\IL^{-}$-model $\seq{W, R, \{S_{w}\}_{w \in W}, \Vdash}$ with $w \Vdash r :\Leftrightarrow w \notin \omega$. 
Let $\Vdash$ be a satisfaction relation on $\mathcal{F}$ such that $v \Vdash q$ and for each $n \in \omega$, $n \nVdash q$. 

\begin{cl}\label{15cl1}
For any formula $A$ with $v(A) \subseteq \{q\}$, there exists $n \in \omega$ such that 
\[
	\forall m \geq n\, (m \Vdash A) \ \text{or} \ \forall m \geq n\, (m \nVdash A). 
\]
\end{cl}

\begin{proof}
We prove by induction on the construction of $A$. 
We only prove the case of $A \equiv B \rhd C$. 
We distinguish the following three cases.

\begin{itemize}
\item (Case 1): There exists an even number $k$ such that $k \Vdash B$, for all $j \leq k$,  $j \nVdash C$ and $v \nVdash C$: 
Let $m \geq k+1$. 
Then $m R k$ and $k \Vdash B$. 
For any $w \in W$ which satisfies $k S_{m} w$, since either $w \leq k$ or $w = v$, we obtain $w \nVdash C$. 
Therefore, $m \nVdash B \rhd C$. 

\item (Case 2): There exists an odd number $k$ such that $k \Vdash B$ and for all $j \leq k$, $j \nVdash C$: 
Let $m \geq k+1$. 
Then $m R k$ and $k \Vdash B$. 
For any $w \in W$ which satisfies $k S_{m} w$, $w \nVdash C$ because $w \leq k$. 
Therefore, $m \nVdash B \rhd C$. 

\item (Case 3): Otherwise: Then, the following conditions (I) and (II) are fulfilled. 
\begin{itemize}
	\item[(I)] For any even number $k$, if $k \Vdash B$, then there exists $j \leq k$ such that $j \Vdash C$ or $v \Vdash C$. 
	\item[(II)] For any odd number $k$, if $k \Vdash B$, then there exists $j \leq k$ such that $j \Vdash C$. 
\end{itemize}

By induction hypothesis, there exists an $n_{0} \in \omega$ such that $\forall m \geq n_{0}\, (m \Vdash B)$ or $\forall m \geq n_{0}\, (m \nVdash B)$. 
We may assume that $n_{0}$ is an odd number. 
%By the definition of $R$, $k$ is an natural number. 
We distinguish the following two cases. 

\begin{itemize}
	\item $\forall m \geq n_{0}\, (m \Vdash B)$: 
Let $m \geq n_{0} + 1$ and $k$ be any element in $W$ satisfying $m R k$ and $k \Vdash B$. 
Since $n_{0}$ is odd and $n_{0} \Vdash B$, there exists a $j_{0} \leq n_{0}$ such that $j_{0} \Vdash C$ by (II). 
We distinguish the following three cases. 
\begin{itemize}
	\item $k$ is odd: By (II), there exists a $j \leq k$ such that $j \Vdash C$. 
	Then $k S_{m} j$ and $j \Vdash C$. 
	\item $k$ is even and $k \geq n_{0}$: Since $k \geq j_{0}$, we have $k S_{m} j_{0}$ and $j_{0} \Vdash C$. 
	\item $k$ is even and $k < n_{0}$: By (I), there exists $j \leq k$ such that $j \Vdash C$ or $v \Vdash C$. 
Since $k < n_{0} \leq m -1$, we obtain $k < m-1$. 
Hence, $k S_{m} j$ and $k S_{m} v$. 
\end{itemize}
In any case, there exists a $w \in W$ such that $k S_{m} w$ and $w \Vdash C$. 
Therefore, $m \Vdash B \rhd C$. 

\item $\forall m \geq n_{0}\, (m \nVdash B)$: 
Let $m \geq n_{0} + 1$ and $k$ be any element in $W$ satisfying $m R k$ and $k \Vdash B$. 
Then $k < n_{0}$ because $k \Vdash B$. 
We distinguish the following two cases. 
\begin{itemize}
	\item $k$ is odd: Since there exists a $j \leq k$ such that $j \Vdash C$ by (II), $k S_{m} j$ and $j \Vdash C$. 
	\item $k$ is even: By (I), there exists a $j \leq k$ such that $j \Vdash C$ or $v \Vdash C$. 
Since $k < n_{0} \leq m -1$, we obtain $k < m-1$ and hence $k S_{m} j$ and $k S_{m} v$. 
\end{itemize}
\end{itemize}
In any case, there exists a $w \in W$ such that $k S_{m} w$ and $w \Vdash C$. 
Therefore, $m \Vdash B \rhd C$. 

\end{itemize}

\end{proof}

We suppose, towards a contradiction, that there exists a formula $A$ such that $v(A) \subseteq \{q\}$ and $\IL^{-}(\J{1}, \J{5}) \vdash A \leftrightarrow A \rhd q$. 
Since $\IL^-(\J{1}, \J{5})$ is valid in $\mathcal{F}$, $A \leftrightarrow A \rhd q$ is also valid in $\mathcal{F}$. 
Then the following claim holds. 

\begin{cl}
For any $n \in \omega$, $n$ is even if and only if $n \Vdash A$. 
\end{cl}

\begin{proof}
We prove by induction on $n$. 

For $n = 0$, since obviously $0 \Vdash A \rhd q$, we have $0 \Vdash A$. 
Suppose $n > 0$ and the claim holds for any natural number less than $n$. 

$(\Leftarrow)$: Assume that $n$ is odd. 
Then $n R n-1$ and since $n-1$ is even, $n-1 \Vdash A$ by induction hypothesis. 
Let $w$ be the any element in $W$ which satisfies $n-1 S_{n} w$. 
By the definition of $S_{n}$, $w \leq n-1$ and hence $w \nVdash q$. 
Therefore $n \nVdash A \rhd q$, and thus $n \nVdash A$. 

$(\Rightarrow)$: Assume that $n$ is even. 
Let $m$ be the any element in $W$ which satisfies $n R m$ and $m \Vdash A$. 
%By the definition of $R$, $m$ is a natural number. 
By induction hypothesis, $m$ is even and hence $m < n-1$. 
Then $m S_{n} v$ and $v \Vdash q$. 
Therefore $n \Vdash A \rhd q$ and hence, $n \Vdash A$. 
\end{proof}

This contradicts Claim \ref{15cl1}. 
Threfore, for any formula $A$ with $v(A) \subseteq \{q\}$, we obtain $\IL^{-}(\J{1}, \J{5}) \nvdash A \leftrightarrow A \rhd q$. 
\end{proof}

As in Corollary \ref{Cor:15}, we obtain the following corollary. 

\begin{cor}
Let $L$ be any logic such that $\IL^{-} \subseteq L \subseteq \IL^-(\J{1}, \J{5})$. 
Then $L$ has neither $\ell$FPP nor CIP. 
\end{cor}

\subsection{A counter model of FPP for $\IL^-(\J{1}, \J{4}_+, \J{5})$}

In Theorems \ref{c3} and \ref{c2}, we proved that the logics $\CL$ and $\IL^-(\J{1}, \J{5})$ do not have $\ell$FPP. 
On the other hand, we proved in Theorem \ref{Thm:lFPP} that $\IL^-(\J{4}, \J{5})$ has $\ell$FPP. 
Thus we cannot provide a counter model of $\ell$FPP for extensions of $\IL^-(\J{4}, \J{5})$. 
In this subsection, we prove that the logics $\IL^-(\J{4}_+, \J{5})$ and $\IL^-(\J{1}, \J{4}_+, \J{5})$ have neither FPP nor CIP. 

\begin{thm}\label{c1}
The formula $\top \rhd \neg p$ has no fixed point in $\IL^{-}(\J{1}, \J{4}_+, \J{5})$. 
That is, for any formula $A$ with $v(A) = \emptyset$, 
\[
	\IL^{-}(\J{1}, \J{4}_{+}, \J{5}) \nvdash A \leftrightarrow \top \rhd \lnot A.
\]
\end{thm}

\begin{proof}
We define an $\IL^-$-frame $\mathcal{F} = \seq{W, R, \{S_{w}\}_{w \in W}}$ as follows: 
\begin{itemize}
	\item $W: = \omega$;
	\item $x R y : \iff x > y$;
	\item For each $n \in W$, $S_{n}:=\{\seq{x, y} \in W^2 : x, y < n$ and $(x \geq y$ or $(x =0$ and $(y$ is even or $y = n-1))) \}$. 
\end{itemize}

%By the definition of $\mathcal{F}$, $R$ is transitive and converse well-founded and for each $w, x, y \in W$, if $x S_{w} y$, then $w R x$ because $w > x$. Thus,  $\mathcal{F}$ is an $\IL^{-}$-frame. 

%\begin{cl}\label{145FC}
%The axioms of $\IL^{-}(\J{1}, \J{4}_{+}, \J{5})$ are valid in $\mathcal{F}$.
%\end{cl}

%\begin{proof}

%We confirm the following properties.

%\begin{enumerate}
%	\item For any $w, x \in W$, $wRx \Rightarrow x S_{w} x$.
%	\item For any $w, x, y \in W$, $x S_{w} y \Rightarrow wRy$.
%	\item For any $w, x, y \in W$, $w R x R y \Rightarrow xS_{w} y$.
%\end{enumerate}

%\noindent
%1. Suppose $w, x \in W$ and $w R x$.  Then $x < w$ by the definition of $R$. Since $x \leq x$, we obtain $x S_{w} x$ 

%\vspace{2mm}
%\noindent
%2.  Suppose $w, x, y \in W$ and $x S_{w} y$.  Then $y < w$ by the definition of $S_{w}$. Therefore, $w R y$. 

%\vspace{2mm}
%\noindent
%3. Suppose $w, x, y \in W$ and $w R x R y$. Then $y < x < w$ by the definition of $R$. Therefore, $x S_{w} y$. 

%By Proposition \ref{Prop:FC}.1, \ref{Prop:FC}.3 and \ref{Prop:FC}.4, the axioms of $\IL^{-}(\J{1}, \J{4}_{+}, \J{5})$ are valid in $\mathcal{F}$. 
%\end{proof}

We draw the relations $S_3$ and $S_4$. 
As in the proof of Theorem \ref{c2}, in the case of $x R y$ for $x, y < n$, $x S_n y$ holds, and the corresponding broken lines are omitted in the figure (Figure \ref{Fig5}). 

\begin{figure}[th]
\centering
\begin{tikzpicture}

\draw (0, 6) node {$S_{3}$};

\node [draw, circle] (4) at (0,0) {$4$};
\node [draw, circle] (3) at (0,1.2) {$3$};
\node [draw, circle] (2) at (0,2.4) {$2$};
\node [draw, circle] (1) at (0,3.6) {$1$};
\node [draw, circle] (0) at (0,4.8) {$0$};

\draw [thick, ->] (4)--(3);
\draw [thick, ->] (3)--(2);
\draw [thick, ->] (2)--(1);
\draw [thick, ->] (1)--(0);
\draw [thick, ->] (0, -1)--(4);

\draw [thick, ->, dashed] (-0.2, 5) arc (45:315:0.3);
\draw [thick, ->, dashed] (-0.2, 3.8) arc (45:315:0.3);
\draw [thick, ->, dashed] (-0.2, 2.6) arc (45:315:0.3);
\draw [thick, ->, dashed] (0) to [out=-45, in=45] (2);

\draw (5, 6) node {$S_{4}$};

\node [draw, circle] (4') at (5,0) {$4$};
\node [draw, circle] (3') at (5,1.2) {$3$};
\node [draw, circle] (2') at (5,2.4) {$2$};
\node [draw, circle] (1') at (5,3.6) {$1$};
\node [draw, circle] (0') at (5,4.8) {$0$};

\draw [thick, ->] (4')--(3');
\draw [thick, ->] (3')--(2');
\draw [thick, ->] (2')--(1');
\draw [thick, ->] (1')--(0');
\draw [thick, ->] (5, -1)--(4');

\draw [thick, ->, dashed] (4.8, 5) arc (45:315:0.3);
\draw [thick, ->, dashed] (4.8, 3.8) arc (45:315:0.3);
\draw [thick, ->, dashed] (4.8, 2.6) arc (45:315:0.3);
\draw [thick, ->, dashed] (4.8, 1.4) arc (45:315:0.3);
\draw [thick, ->, dashed] (0') to [out=-45, in=45] (2');
\draw [thick, ->, dashed] (0') to [out=-30, in=30] (3');

\end{tikzpicture}
\caption{A counter model of FPP for $\IL^-(\J{1}, \J{4}_+, \J{5})$}\label{Fig5}
\end{figure}

Then $\IL^{-}(\J{1}, \J{4}_{+}, \J{5})$ is valid in $\mathcal{F}$. 
%We arbitrarily fix a binary relation $\Vdash$ and define the $\IL^{-}$-model $\seq{W, R, \{S_{x}\}_{x \in W}, \Vdash}$. 
Let $\Vdash$ be an arbitrary satisfaction relation on $\mathcal{F}$. 

\begin{cl}\label{145cl1}
For any formula $A$ with $v(A) = \emptyset$, there exists an $n \in W$ such that
\[
	\forall m \geq n\, (m \Vdash A) \ \text{or} \ \forall m \geq n\, (m \nVdash A).
\]
\end{cl}

\begin{proof}
This is proved by induction on the construction of $A$.
We prove only the case of $A \equiv B \rhd C$. 
We distinguish the following three cases.
\begin{itemize}
\item (Case 1): There exists an $n > 0$ such that $n \Vdash B$ and for all $k \leq n$, $k \nVdash C$. 
Let $m \geq n+1$. 
Then $m R n$ and $n \Vdash B$. 
Also, for any $k \in W$, if $nS_{m} k$, then $k \leq n$ because $n \neq 0$. 
Therefore $k \nVdash C$. 
Thus, $m \nVdash B \rhd C$.

\item (Case 2): $0 \Vdash B$ and for all even numbers $k$, $k \nVdash C$. 
By induction hypothesis, there exists an $n_{0} \in W$ such that $\forall m \geq n_{0}\, (m \Vdash C)$ or $\forall m \geq n_{0}\, (m \nVdash C)$. 
Since there are infinitely many even numbers $k \in W$ such that $k \nVdash C$, 
we obtain $\forall m \geq n_{0}\, (m \nVdash C)$. 
Then, for any $m \geq n_{0}+1$, $m R 0$ and $0 \Vdash B$. 
Let $k \in W$ be such that $0 S_{m} k$. 
Then $k$ is even or $k = m-1$ by the definition of $S_{m}$. 
By our supposition, if $k$ is even, then $k \nVdash C$. 
If $k = m-1$, then $m-1 \nVdash C$ because $m-1 \geq n_{0}$. 
Therefore, in either case, $k \nVdash C$. 
Thus $m \nVdash B \rhd C$. 

\item (Case 3): Otherwise: Then, the following conditions (I) and (II) are fulfilled. 
\begin{itemize}
	\item[(I)] For any $n > 0$, if $n \Vdash B$, then there exists a $k \in W$ such that $k \leq n$ and $k \Vdash C$. 
	\item[(II)] If $0 \Vdash B$, then there exists an even number $k \in W$ such that $k \Vdash C$. 
\end{itemize}

We distinguish the following two cases. 
\begin{itemize}
	\item $0 \nVdash B$: Let $m \geq 0$. 
For any $n \in W$ satisfying $m R n$ and $n \Vdash B$, since $n \neq 0$, there exists a $k \leq n$ such that $k \Vdash C$ by the condition (I). 
Since $n S_{m} k$, we obtain $m \Vdash B \rhd C$.

	\item $0 \Vdash B$: By the condition (II), there exists an even number $k$ such that $k \Vdash C$. 
Let $m \geq k+1$ and let $n \in W$ be such that $m R n$ and $n \Vdash B$. 
If $n \neq 0$, then there exists a $k' \leq n$ such that $k' \Vdash C$ and $n S_{m}k'$ by the condition (I). 
If $n =0$, then since $k$ is even and $k < m$, we obtain $n S_{m} k$ and $k \Vdash C$. Therefore $m \Vdash B \rhd C$. 
\end{itemize}
\end{itemize}
\end{proof}

We suppose, towards a contradiction, that there exists a formula $A$ such that $v(A) = \emptyset$ and $\IL^{-}(\J{1}, \J{4}_{+}, \J{5}) \vdash A \leftrightarrow \top \rhd \lnot A$. 
Then $A \leftrightarrow \top \rhd \lnot A$ is valid in $\mathcal{F}$ because so is $\IL^{-}(\J{1}, \J{4}_{+}, \J{5})$. 
Then the following claim holds. 

\begin{cl}
For any $n \in W$, $n$ even if and only if $n \Vdash A$. 
\end{cl}

\begin{proof}
We prove by induction on $n$. 
For $n = 0$, obviously $0 \Vdash A$. 
Suppose $n > 0$ and the claim holds for any natural number less than $n$. 

$(\Leftarrow)$: Assume that $n$ is odd. 
Then $n R 0$. 
For any $k \in W$ which satisfies $0 S_{n} k$, since $n$ is odd, $k$ is even and $k < n$. 
By induction hypothesis, $k \Vdash A$. 
Thus, we obtain $n \nVdash \top \rhd \lnot A$ and hence, $n \nVdash A$. 

$(\Rightarrow)$: Assume that $n$ is even. 
Let $m \in W$ be such that $n R m$. 
We distinguish the following three cases.
\begin{itemize}
		\item $m=0$: Then $0 S_{n} n-1$. 
Since $n-1$ is odd, $n-1 \Vdash \lnot A$ by induction hypothesis.

		\item $m$ is even and $m \neq 0$: Then $m S_{n} m-1$. 
Since $m-1$ is odd, $m-1 \Vdash \lnot A$ by induction hypothesis.

		\item $m$ is odd: Then $m S_{n} m$. 
Since $m$ is odd, $m \Vdash \lnot A$ by induction hypothesis.
\end{itemize}
In any case, there exists a $w \in W$ such that $m S_{n} w$ and $w \Vdash \lnot A$. 
Therefore, we obtain $n \Vdash \top \rhd \lnot A$ and hence, $n \Vdash A$.
\end{proof}

This contradictions Claim \ref{145cl1}. 
Therefore, there is no formula $A$ such that $v(A) = \emptyset$ and $\IL^{-}(\J{1}, \J{4}_{+}, \J{5}) \nvdash A \leftrightarrow \top \rhd \lnot A$. 
\end{proof}

\begin{cor}
Every sublogic of $\IL^{-}(\J{1}, \J{4}_{+}, \J{5})$ does not have FPP. 
Furthermore, if $\IL^{-}(\J{4}_{+}) \subseteq L \subseteq \IL^{-}(\J{1}, \J{4}_{+}, \J{5})$, then $L$ does not have CIP. 
\end{cor}

\begin{proof}
By Theorem \ref{c1}, every sublogic of $\IL^{-}(\J{1}, \J{4}_{+}, \J{5})$ does not have FPP. 
By Lemma \ref{FPP-CU}, every logic $L$ with $\IL^{-}(\J{4}_{+}) \subseteq L \subseteq \IL^{-}(\J{1}, \J{4}_{+}, \J{5})$ does not have CIP. 
\end{proof}

\section{Concluding remarks}

In this paper, we provided a complete description of twelve sublogics of $\IL$ concerning UFP, FPP and CIP. 
In particular, for these sublogics $L$, we proved that $L$ has FPP if and only if $L$ contains $\IL^-(\J{2}_+, \J{5})$. 
On the other hand, there are many other logics between $\IL^-$ and $\IL$. 
For instance, Kurahashi and Okawa \cite{KO20} introduced eight sublogics such as $\IL^-(\J{2}, \J{4}_+, \J{5})$ that are not in Figure \ref{Fig1}, and proved that these eight logics are not complete with respect to regular Veltman semantics but complete with respect to generalized Veltman semantics. 
Then it is natural to investigate a sharper threshold for FPP in a larger class of sublogics. 
Then for example, we propose a question if $\J{2}_+$ can be weakened by $\J{2}$ in the statement of Corollary \ref{Cor:FPP}. 

\begin{prob}
Does the logic $\IL^-(\J{2}, \J{4}_+, \J{5})$ have FPP?
\end{prob}

In our proofs of Theorem \ref{T3-1}, Theorem \ref{FPrhd} and Theorem \ref{Thm:lFPP}, the use of the axiom scheme $\J{5}$ seems inevitable. 
In fact, $\CL$ ($= \IL^-(\J{1}, \J{2}_+)$) fails to have $\ell$FPP. 
Thus we propose a question whether $\J{5}$ is necessary or not for $\ell$FPP and FPP. 
For this question, we keep in mind the fact that an extension $L$ of $\mathbf{K4}$ proves the axiom scheme $\G{3}$ if $L$ has FPP.

\begin{prob}\leavevmode
\begin{enumerate}
	\item For every extension $L$ of $\IL^-(\J{2}_+)$, if $L$ has FPP, then does $L$ prove $\J{5}$?
	\item For every extension $L$ of $\IL^-(\J{4})$, if $L$ has $\ell$FPP, then does $L$ prove $\J{5}$?
\end{enumerate}
\end{prob}

\bibliographystyle{plain}
\bibliography{ref}

\begin{thebibliography}{10}

\bibitem{AHD01}
Carlos Areces, Eva Hoogland, and Dick de~Jongh.
\newblock Interpolation, definability and fixed points in interpretability
  logics.
\newblock In {\em Advances in modal logic, Vol. 2 (Uppsala, 1998)}, volume 119
  of {\em CSLI Lecture Notes}, pages 35--58, Stanford, CA, 2001. CSLI Publ.

\bibitem{Ber76}
Claudio Bernardi.
\newblock The uniqueness of the fixed-point in every diagonalizable algebra.
\newblock {\em Studia Logica}, 35(4):335--343, 1976.

\bibitem{Boo79}
George Boolos.
\newblock {\em The unprovability of consistency. An essay in modal logic}.
\newblock Cambridge University Press, Cambridge, 1979.

\bibitem{Boo93}
George Boolos.
\newblock {\em The {L}ogic of {P}rovability}.
\newblock Cambridge University Press, Cambridge, 1993.

\bibitem{DeJVel90}
Dick de~Jongh and Frank Veltman.
\newblock Provability logics for relative interpretability.
\newblock In P.P. Petkov, editor, {\em Mathematical logic}, pages 31--42.
  Plenum Press, New York, 1990.

\bibitem{DeJVis91}
Dick de~Jongh and Albert Visser.
\newblock Explicit fixed points in interpretability logic.
\newblock {\em Studia Logica}, 50(1):39--49, 1991.

\bibitem{Ign91}
Konstantin~N. Ignatiev.
\newblock Partial conservativity and modal logics.
\newblock {\em ITLI Publication Series X-91-04}, 1991.

\bibitem{KO20}
Taishi Kurahashi and Yuya Okawa.
\newblock Modal completeness of sublogics of the interpretability logic
  {$\mathbf{IL}$}.
\newblock Submitted. arXiv:2004.03813.

\bibitem{Sam76}
Giovanni Sambin.
\newblock An effective fixed-point theorem in intuitionistic diagonalizable
  algebras.
\newblock {\em Studia Logica}, 35(4):345--361, 1976.

\bibitem{Smo78}
Craig Smory{\'n}ski.
\newblock Beth's theorem and self-referential sentences.
\newblock In L.~Pacholski A.~Macintyre and J.~Paris, editors, {\em Logic
  Colloquium '77 (Proc. Conf., Wroc{\l}aw, 1977)}, volume~96 of {\em Studies in
  Logic and the Foundations of Mathematics}, pages 253--261, 1978.

\bibitem{Smo85}
Craig Smory{\'n}ski.
\newblock {\em Self-reference and modal logic}.
\newblock Universitext. Springer-Verlag, New York, 1985.

\bibitem{Vis88}
Albert Visser.
\newblock Preliminary notes on interpretability logic.
\newblock Technical Report~29, Department of Philosophy, Utrecht University,
  1988.

\bibitem{Vis90}
Albert Visser.
\newblock Interpretability logic.
\newblock In P.P. Petkov, editor, {\em Mathematical Logic}, pages 175--208.
  Plenum Press, New York, 1990.

\bibitem{Vis97}
Albert Visser.
\newblock An overview of interpretability logic.
\newblock In {\em Advances in modal logic}, volume~1, pages 307--359, 1997.

\end{thebibliography}

\end{document}